\documentclass[a4paper, 11 pt,dvipsnames]{amsart}

\usepackage[dvipsnames]{xcolor}

\usepackage[leqno]{amsmath}
\usepackage{amsfonts}
\usepackage{amssymb}
\usepackage[british,UKenglish,USenglish,english,american]{babel}
\usepackage{amsthm}
\usepackage[all, cmtip]{xy}
\usepackage{tikz-cd}
\usepackage{mathrsfs}
\usepackage{textcomp}
\usepackage{mathtools}
\usepackage{enumerate}
\usepackage{color}
\usepackage{fullpage}
\usepackage{caption}
\usepackage{latexsym}
\usepackage{stmaryrd}
\usepackage[utf8]{inputenc}

\usepackage{graphicx}

\usepackage{hyperref}

 \hypersetup{
 colorlinks,
 citecolor=green,
 linkcolor=blue,
 urlcolor=blue}

\UseRawInputEncoding

\newcommand{\un}{\underline}
\newcommand{\ov}{\overline}
\newcommand{\wt}{\widetilde}
\newcommand{\wh}{\widehat}

\newcommand{\bN}{\mathbb{N}}
\newcommand{\bZ}{\mathbb{Z}}
\newcommand{\bF}{\mathbb{F}}
\newcommand{\bE}{\mathbb{E}}
\newcommand{\bP}{\mathbb{P}}

\newcommand{\cE}{\mathcal{E}}
\newcommand{\cF}{\mathcal{F}}
\newcommand{\cG}{\mathcal{G}}
\newcommand{\cH}{\mathcal{H}}
\newcommand{\cI}{\mathcal{I}}
\newcommand{\cJ}{\mathcal{J}}
\newcommand{\cK}{\mathcal{K}}
\newcommand{\cL}{\mathcal{L}}
\newcommand{\cO}{\mathcal{O}}
\newcommand{\cT}{\mathcal{T}}
\newcommand{\cM}{\mathcal{M}}
\newcommand{\cN}{\mathcal{N}}
\newcommand{\cS}{\mathcal{S}}
\newcommand{\cD}{\mathcal{D}}

\newcommand{\sE}{{\mathscr E}}
\newcommand{\sF}{{\mathscr F}}
\newcommand{\sG}{{\mathscr G}}

\newcommand{\sI}{{\mathscr I}}
\newcommand{\sK}{{\mathscr K}}

\newcommand{\sN}{{\mathscr N}}

\newcommand{\coker}[1]{\operatorname{coker}(#1)}
\newcommand{\Deg}[1]{\operatorname{Deg}(#1)}
\newcommand{\Ext}[1]{\operatorname{Ext}^{#1}}
\newcommand{\cExt}[1]{\operatorname{{\mathcal Ext}}^{#1}}
\newcommand{\cEnd}{\operatorname{{\mathcal End}}}
\newcommand{\cHom}{\operatorname{{\mathcal Hom}}}
\renewcommand{\ker}[1]{\operatorname{ker}(#1)}
\newcommand{\Gr}[2]{\operatorname{Gr}_{#1}(#2)}
\newcommand{\Hom}{\operatorname{Hom}}
\newcommand{\im}[1]{\operatorname{im}(#1)}

\newcommand{\rk}[1]{\operatorname{rk}(#1)}
\newcommand{\Rk}[1]{\operatorname{R}(#1)}
\newcommand{\Spec}[1]{\operatorname{Spec}(#1)}
\newcommand{\supp}{\operatorname{supp}}
\newcommand{\red}{\operatorname{red}}
\newcommand{\Coh}[2]{{\operatorname{Coh}}_{#1}^{#2}}
\newcommand{\Gm}{\mathbb{G}_m}
\newcommand{\Bl}{\operatorname{Bl}}
\newcommand{\Tot}{\operatorname{Tot}}
\renewcommand{\Im}{\operatorname{Im}}
\newcommand{\Inf}{\operatorname{Inf}}
\newcommand{\Tors}{\operatorname{Tors}}
\newcommand{\res}{\operatorname{res}}

\newcommand{\sHom}[1]{\un{\operatorname{Hom}}(#1)}

\theoremstyle{plain}
\newtheorem{thm}{Theorem}[section]
\newtheorem{lemma}[thm]{Lemma}

\newtheorem{prop}[thm]{Proposition}
\newtheorem{cor}[thm]{Corollary}
\newtheorem{fact}[thm]{Fact}

\newtheorem*{ThmA}{Theorem A}
\newtheorem*{ThmB}{Theorem B}
\newtheorem*{ThmC}{Theorem C}
\newtheorem*{ThmD}{Theorem D}

\theoremstyle{remark}
\newtheorem{rmk}[thm]{Remark}
\newtheorem{ex}[thm]{Example}

\theoremstyle{definition}
\newtheorem{defi}[thm]{Definition}

\numberwithin{equation}{section}

\newenvironment{sis}{\left\{\begin{aligned}}{\end{aligned}\right.}

\newcommand{\surj}{\twoheadrightarrow}

\newcommand{\Xred}{X_{\mathrm{red}}}

\begin{document}

\title{Moduli of  sheaves on ribbons}

\author{Michele Savarese}
\email{micsavarese@libero.it}

\author{Filippo Viviani}
\address{Filippo Viviani, Dipartimento di Matematica, Universit\`a di Roma ``Tor Vergata'', Via della Ricerca Scientifica 1, 00133 Roma, Italy}
\email{viviani@mat.uniroma2.it}


\begin{abstract}
We study the geometry of the moduli stack of torsion-free sheaves on ribbons. We introduce a stratification of the stack by the complete type of the sheaves, and we investigate the geometric properties of the strata and their closure relation, and which strata intersect the (semi)stable locus. Then we describe the irreducible components of the stack, by revealing an interesting trichotomy between Fano, Calabi-Yau and canonically polarized cases. Finally, we compute the tangent space of the moduli stack at a given sheaf. 
\end{abstract}

\maketitle


\section*{Introduction}

This paper is devoted to the study of moduli  of coherent sheaves on ribbons. A \emph{ribbon} $X$ over an algebraically closed field $k$ is an irreducible and projective curve that is locally (but not necessarily globally) isomorphic to $2C\subset S$, where $C$ is a smooth curve inside a smooth surface $S$ (see \cite{BE} for a classification of ribbons).

\subsection*{Motivation}

The motivation of this paper is to understand the geometry of the nilpotent cone of the stack of Higgs pairs, which we now briefly recall.

Let $C$ be an irreducible smooth projective curve over $k=\ov k$ and let $L$ be a line bundle on $C$. Fix two integers $r\geq 1$ and $d\in \bZ$. Consider the stack $\cM_{(C,L)}(r,d)$ parametrizing \emph{$L$-twisted Higgs pairs} on $C$ of rank $r$ and degree $d$, i.e. pairs $(E,\phi)$ consisting of a vector bundle of rank $r$ and degree $d$ on $C$ together with a twisted endomorphism $\phi:E\to E\otimes L$ (called the Higgs field). The \emph{nilpotent cone} stack $\cN_{(C,L)}(r,d)$ is the closed substack parametrizing Higgs pairs $(E,\phi)$ such that $\phi$ is a nilpotent twisted endomorphism, and it plays a crucial role in studying the geometry of the stack $\cM_{(C,L)}(r,d)$, see e.g. \cite{Hau1, Hau2} for surveys. Via the spectral correspondence, the nilpotent cone  stack $\cN_{(C,L)}(r,d)$ is isomorphic to the  moduli stack $\cM_{r C}(r,d)$  of torsion-free sheaves of generalized rank $r$ and generalized degree $d$ on the primitive multiple curve $rC\subset \rm{Tot}_C(L)$, where $\rm{Tot}_C(L)\to C$ is the total space of the line bundle $L$ and $C$ sits inside $\rm{Tot}_C(L)$ as the zero section. Moreover, the isomorphism $\cN_{(C,L)}(r,d)\cong \cM_{rC}(r,d)$ is compatible with the natural semistability conditions and hence induces an isomorphism of the corresponding moduli spaces.  In the important special case where $L=\omega_C$, the geometry of the nilpotent cone $\cN_{(C,\omega_C)}(r,d)$ has been investigated in \cite{Lau, Gin, Boz, HH}, but very little seem to be known for an arbitrary line bundle $L$. 

This paper is a first step towards understanding of the geometry of $\cM_{rC}(r,d)$. Indeed, by looking at the scheme-theoretic support of the sheaves, we get an increasing sequence of closed substacks
$$
\cM_{C}(r,d)\subset \cM_{2C}(r,d)\subset \ldots \subset \cM_{rC}(r,d).
$$
While the moduli stack $\cM_{C}(r,d)$ parametrizing vector bundles of rank $r$ and degree $d$ on the smooth curve $C$ is classically studied, the goal of this paper is to investigate in full generality the geometry of the moduli stack  $\cM_{2C}(r,d)$ and, more generally, the moduli stack of torsion-free sheaves on arbitrary ribbons. This is the first steps towards the study of the full moduli stack $ \cM_{rC}(r,d)$, which will be the subject of our future work \cite{SV}, and it already reveals the full richness of the geometry of these moduli stacks/spaces, although it is notationally easier to deal with respect to the higher multiplicity case.


\subsection*{Previous results}

There are in the literature some partial results on the moduli of some special types of sheaves on ribbons, which we now review. 

First of all, generalized line bundles (i.e. torsion-free sheaves which are generically invertible), has been introduced by Bayer-Eisenbud in \cite{BE} and then studied by Eisenbud-Green in \cite{EG} in order to deal with the Green conjecture about the Clifford index of smooth curves. 

The moduli space of semistable torsion-free sheaves of generalized rank $2$,  
which are either generalized line bundles or direct images of rank $2$ vector bundles on the underlying reduced curve, has been studied for some special ribbons (namely those whose normal bundle is the canonical line bundle of the reduced curve) by Donagi-Ein-Lazarsfeld in \cite[\S 3]{DEL} and for arbitrary ribbons by Yang \cite{Yan} and Chen-Kass \cite{CK}, with the addendum of \cite{S2} who solved a question left open in loc. cit. 


The study of arbitrary coherent sheaves on ribbons  (or more generally on primitive multiple curves of arbitrary multiplicity) has been undertaken systematically  by Dr\'{e}zet in the  papers \cite{DR1}, \cite{DR1bis}, \cite{DR2}, \cite{DR3}. In the same papers, Dr\'{e}zet studied the geometry (i.e. number of irreducible components, their dimension, computation of the Zariski tangent space) of the open subset of the moduli space of semistable torsion-free sheaves consisting of quasi-locally free sheaves of rigid type (see \ref{sub:qlf-rig}). In the follow-up papers \cite{DR4} and \cite{DR5}, Dr\'{e}zet investigated when the above mentioned components can be obtained as limits of sheaves under a partial smoothing of the ribbon to a nodal curves with two irreducible components. 
The local structure of the open subset of the moduli space of semistable torsion-free sheaves consisting of generalized vector bundles (i.e.  torsion-free sheaves which are locally free away from a finite set of points) has been studied by Inaba \cite{I}. 

Finally, Yuan proved in \cite{Yua} an upper bound for the dimension of the moduli stack of torsion-free sheaves of generalized rank $n$ on a primitive multiple curve of multiplicity $n$.


\subsection*{Our results}

Let $X$ be a ribbon over an algebraically closed field $k=\ov k$. The  underlying reduced curve $\Xred$ is an irreducible, projective smooth curve over $k$ whose genus will be denoted by $\ov g$. The ideal sheaf of $\Xred$ in $X$, called the nilradical and denoted by $\sN$, is (the pushforward of) a line bundle on $\Xred$. The dual line bundle $\sN^*$ is the normal bundle of $\Xred$ in $X$ and we set $\delta:=\deg \sN^*$.

Any coherent sheaf $\sF$ on $X$ comes with an exact sequence (called the first canonical filtration, see Definition \ref{D:CanFiltr})
$$
0 \to \sN \sF\to \sF \to \sF_{|\Xred}\to 0,
$$
where $\sN\sF$ and $\sF_{|\Xred}$ are both coherent sheaves on $\Xred$. 
We define the \emph{complete type} of $\sF$ by 
$$
(r_\bullet(\sF); d_\bullet(\sF))=(r_0(\sF),r_1(\sF);d_0(\sF),d_1(\sF)):=(\rk{\sF_{|\Xred}}, \rk{\sN\sF}; \deg(\sF_{|\Xred}),\deg(\sN\sF)).
$$
The complete type allows to define two numerical (i.e. invariant by deformations) invariants (see \ref{sub:R-D}):  the \emph{generalized rank} $\Rk{\sF}:=r_0(\sF)+r_1(\sF)$,  which is also equal to the length of its generic stalk $\sF_{\eta}$ as an $\cO_{X,\eta}$-module (where $\eta$ denotes the generic point of $X$), and the \emph{generalized degree} $\Deg{\sF}:=d_0(\sF)+d_1(\sF)$.

For any $(R, D)\in \bN_{>0} \times \bZ$,  we denote by  $\cM_X(R,D)$ the stack of torsion-free (i.e. pure of one-dimensional support) sheaves on a ribbon $X$ of generalized rank $R$ and generalized degree $D$. We can stratify $\cM_X(R,D)$ according to the complete type of the sheaves
$$
\cM_X(R,D)=\coprod_{(r_\bullet;d_\bullet)\in \cS(R,D)} \cM_X(r_{\bullet};d_{\bullet}),
$$
where each stratum $\cM_X(r_{\bullet};d_{\bullet})$ parametrizes torsion-free sheaves on $X$ of complete type $(r_\bullet;d_\bullet)$ and 
$$
\cS(R,D):=\{(r_0,r_1;d_0,d_1)\in \bN^2\times \bZ^2\: : r_0+r_1=R \text{ and  } d_0+d_1=D\}.
$$

\vspace{0.1cm}

Our first main result describes the properties of this stratification.

\begin{ThmA}\label{T:ThmA}(see Theorems \ref{T:non-empty}, \ref{T:irr-dim}, \ref{T:specia}) 
Let $X$ be a ribbon and set $\delta:=-\deg \sN$ and $\ov g$ be the genus of $\Xred$. Fix $(R, D)\in \bN_{>0} \times \bZ$. 
\begin{enumerate}[(1)]
\item For any  $(r_\bullet;d_\bullet)\in \cS(R,D)$, we have that 
\begin{equation*}
\cM_X(r_{\bullet};d_{\bullet})\neq \emptyset \Longleftrightarrow \begin{sis} 
&r_0\geq r_1, \\
&r_1=0\Rightarrow d_1=0,\\
& r_0=r_1 \Rightarrow d_0\geq d_1+r_1\delta.
\end{sis}
\end{equation*}
We say that a sequence $(r_\bullet;d_\bullet)\in \cS(R,D)$ is admissible if it is satisfies the above conditions and we denote by  $\cS_{adm}(R,D)$ the subset of admissible elements of $\cS(R,D)$.

\item For any  $(r_\bullet;d_\bullet)\in \cS_{adm}(R,D)$, the stratum $\cM_X(r_{\bullet};d_{\bullet})$ is a locally closed substack  which, endowed with its reduced stack structure, is  irreducible and smooth of dimension equal to 
$$
\dim \cM_X(r_{\bullet};d_{\bullet})=(r_0^2+r_1^2)(\ov g -1)+r_0r_1\delta=R^2(\ov g-1)+r_0r_1[\delta -2(\ov g-1)].
$$
\item Assume that $(r_\bullet;d_\bullet), (\wt r_\bullet;\wt d_\bullet)\in \cS_{adm}(R,D)$.  Then 
$$
\ov{\cM_X(r_\bullet; d_\bullet)}\cap \cM_X(\wt r_{\bullet};\wt d_{\bullet})\neq \emptyset \Longleftrightarrow (r_\bullet;d_\bullet)\geq (\wt r_\bullet;\wt d_\bullet), \text{ i.e. either } r_1> \wt r_1 \: \text{ or } r_1=\wt r_1 \: \text{ and } \: d_1\geq \wt d_1.
$$
\end{enumerate} 
\end{ThmA}

\vspace{0.1cm}

Our second main result is the description of the irreducible components of the stack $\cM_X(R,D)$. There are three cases depending on the comparison between the degree of the normal sheaf $\delta:=-\deg(\sN)$ and the canonical degree $2\ov g-2$ of $\Xred$.

\begin{ThmB}\label{T:ThmB}(see Proposition \ref{P:irrcomp1}, Corollary \ref{C:irrcomp2})
Notation as in Theorem A. 
\noindent
\begin{enumerate}[(1)]
\item If $\delta< 2\ov g-2$ (\emph{canonically polarized} case) then the irreducible components of $\cM_X(R,D)$ are
$$
\{\ov{\cM_X(r_{\bullet};d_{\bullet})}\}_{(r_{\bullet};d_{\bullet})\in \cS_{adm}(R,D)},
$$
and we have that
$$
(r_\bullet;d_\bullet)\geq (\wt r_\bullet;\wt d_\bullet)\Rightarrow \dim \ov{\cM_X(r_\bullet;d_\bullet)}\leq \dim \ov{\cM_X(\wt r_\bullet;\wt d_\bullet)}.
$$
In particular, $\cM_X(R,D)$ has dimension $R^2(\ov g-1)$ and the unique irreducible component of maximal dimension is $\cM_X(R,0;D,0)$, which parametrizes vector bundles on $\Xred$ of rank $R$ and degree $D$. 
\item If $\delta= 2\ov g-2$ (\emph{Calabi-Yau} case) then the irreducible components of $\cM_X(R,D)$ are
$$
\{\ov{\cM_X(r_{\bullet};d_{\bullet})}\}_{(r_{\bullet};d_{\bullet})\in \cS_{adm}(R,D)},
$$
and they have all dimension equal to $R^2(\ov g-1)$.
\item  If $\delta > 2\ov g-2$ (\emph{Fano} case) then the irreducible components of $\cM_X(R,D)$ are
$$
\{\ov{\cM_X(r_{\bullet};d_{\bullet})}\}_{(r_{\bullet};d_{\bullet})\in \cS_{rig}(R,D)},
$$
where the subset  $\cS_{rig}(R,D)\subset \cS_{adm}(R,D)$ of rigid sequences is defined by  
$$
(r_{\bullet};d_{\bullet})\in \cS_{rig}(R,D)\stackrel{def}{\Longleftrightarrow} 
\begin{cases}
r_0=r_1=\frac{R}{2} & \text{ if } R \: \text{ is even,}\\
r_0=r_1+1=\frac{R+1}{2} & \text{ if } R \: \text{ is odd.}\\
\end{cases}
$$
In particular, $\cM_X(R,D)$ has pure dimension equal to 
$$\dim \cM_X(R,D)=
\begin{cases}
\frac{R^2}{4}(2\ov g-2+\delta) & \text{ if } R \text{ is even,} \\
\frac{R^2+1}{4}(2\ov g-2)+\frac{R^2-1}{4}\delta  & \text{ if } R \text{ is odd.} 
\end{cases}
$$
\end{enumerate}
\end{ThmB} 
The above Theorem B answers completely the open question of \cite[Sec. 1.3.1]{DR2}, and it generalizes it from the semistable locus to the entire stack of torsion-free sheaves. 

The terminology of  the three cases of Theorem B is due to the following fact: if $X$ embeds into an irreducible smooth projective $S$ (so that $X=2\Xred\subset S$) and $S$ is canonically polarized (resp. Calabi-Yau, resp. Fano) then the adjunction formula $K_{\Xred}=(K_S+\Xred)_{|\Xred}$ gives that $\delta< 2\ov g-2$ (resp. $\delta= 2\ov g-2$, resp. $\delta> 2\ov g-2$).

For the special ribbon $X:=2\Xred\subset \Tot_{\Xred}(\omega_{\Xred})$ (for which the normal bundle is $\omega_{\Xred}$ and hence $\delta=2\ov g-2$), the result of the above Theorem B(2) can be deduced from \cite{Lau, Boz}.

\vspace{0.1cm}

We now determine which strata of $\cM_X(R,D)$ intersect the open substacks $\cM_X(R,D)^{ss}$ and $\cM_X(R,D)^s$ of, respectively, semistable and stable sheaves.  Recall that, by defining the slope of a torsion-free sheaf $\sF$ on $X$  to be 
$$\mu(\sF):= \frac{\Deg{\sF}}{\Rk{\sF}},$$
a torsion-free sheaf $\sF$ on $X$ is said to be semistable (resp. stable) if for any proper subsheaf  $0\neq \sG\subsetneq \sF$ we have that 
$$\mu(\sG)\le\mu(\sF) \: \text{ (resp. } \mu(\sG)<\mu(\sF)).$$
As we recall in Fact \ref{F:ModSpa}, the locus $\cM_X(R,D)^{ss}$ of semistable sheaves on $\cM_X(R,D)$ is an open substack of finite type over $k$ and its admits a projective adequate (or good if $\rm{char}(k)=0$) moduli space $\Phi^{ss}: \cM_X(R,D)^{ss}\to M_X(R,D)^{ss}$, that identifies semistable sheaves that have the same Jordan-Holder graduation. The restriction of $\Phi^{ss}$ to the  open substack $\cM_X(R,D)^{s}$ of stable sheaves is a $\Gm$-gerbe. 

\begin{ThmC}\label{T:ThmC}(see Theorem \ref{T:ss-strata})
Let $(r_\bullet;d_\bullet)\in \cS_{adm}(R,D)$.
\begin{enumerate}[(1)]
\item  If the stratum $\cM_X(r_\bullet; d_\bullet)$ intersects $\cM_X(R,D)^{ss}$ then one the following occurs:
\begin{enumerate}[(a)]
\item $r_1=0$;
\item $r_0>r_1>0$ and $\displaystyle \frac{d_0-(r_0+r_1)\delta}{r_0}\le\frac{d_1}{r_1}\le\frac{d_0}{r_0}$;
\item $r_0=r_1$ and $d_0\leq d_1+2r_0\delta$.
\end{enumerate}
\vspace{0.1cm} 
If $\ov g\geq 1$ then the above numerical conditions are also sufficient for the non-emptiness of $\cM_X(r_\bullet; d_\bullet)\cap \cM_X(R,D)^{ss}$.
\item  If the stratum $\cM_X(r_\bullet; d_\bullet)$ intersects $\cM_X(R,D)^{s}$ then one the following occurs:
\begin{enumerate}[(a)]
\item $r_1=0$;
\item $r_0>r_1>0$ and $\displaystyle \frac{d_0-(r_0+r_1)\delta}{r_0}<\frac{d_1}{r_1}<\frac{d_0}{r_0}$;
\item $r_0=r_1$ and $d_0< d_1+2r_0\delta$.
\end{enumerate}
\vspace{0.1cm} 
If $\ov g\geq 2$ then the above numerical conditions are also sufficient for the non-emptiness of $\cM_X(r_\bullet; d_\bullet)\cap \cM_X(R,D)^{s}$.
\end{enumerate}
\end{ThmC}
The above Theorem C answers completely the open question of \cite[Sec. 1.3.2]{DR2}.

As an example, we study in subsection \ref{sub:ss(n,1)} the semistability of torsion-free sheaves of type $(n,1)$: we show that these sheaves can be obtained as pushforward of quasi-locally free sheaves of the same type on some blowups of the ribbon (see Proposition \ref{P:BlowUp}) and we prove a refined criterion for (semi)stability (see Proposition \ref{P:ss(n,1)}).

\vspace{0.1cm}

Our final result is the computation of the dimension of the tangent space $T_{\sF}\cM_X(R,D)$ of $\cM_X(R,D)$ at a sheaf $\sF\in \cM_X(r_\bullet;d_\bullet)$. Recall that the tangent space $T_{\sF}\cM_X(R,D)$ parametrizes infinitesimal deformations of $\sF$ and it is therefore equal to $\Ext{1}(\sF,\sF)$. The tangent space  $T_{\sF}\cM_X(R,D)$ admits two notable subspaces
\begin{equation}\label{E:inc-T*}
H^1(X,\cExt{0}(\sF,\sF))\subseteq T_{\sF}\cM_X(r_\bullet;d_\bullet)\subseteq   T_{\sF}\cM_X(R,D)=\Ext{1}(\sF,\sF),
\end{equation}
where $H^1(X,\cExt{0}(\sF,\sF))$ parametrizes infinitesimal deformations of $\sF$ that are locally trivial and $T_{\sF}\cM_X(r_\bullet;d_\bullet)$ is the tangent space of $\sF$ at the stratum $\cM_X(r_\bullet;d_\bullet)$. 
Since the stratum $\cM_X(r_\bullet;d_\bullet)$ is smooth by Theorem A, the dimension of its tangent space at $\sF$ is equal to
\begin{equation}\label{E:Tan-strat}
\dim T_{\sF}\cM_X(r_\bullet;d_\bullet)=\dim \cM_X(r_\bullet;d_\bullet)+\dim \Inf_{\sF}\cM_X(r_\bullet;d_\bullet),
\end{equation}
where $ \Inf_{\sF}\cM_X(r_\bullet;d_\bullet)$ is the infinitesimal automorphism  space of $\sF$ which is equal to $\Ext{0}(\sF,\sF)$ and $\dim  \cM_X(r_\bullet;d_\bullet)$ is computed in Theorem A.
Therefore, we can regard the dimension of the space  $\dim T_{\sF}\cM_X(r_\bullet;d_\bullet)$ as "known" and we determine the codimension of the two inclusions in \eqref{E:inc-T*}. 

The answer depends on the local structure of $\sF$, as we now describe. Fact  \ref{F:structure} implies that the stalk of $\sF\in \cM_X(r_\bullet;d_\bullet)$ at a point $p$ is isomorphic to 
$$
\sF_p= \cO_{\Xred,p}^{\oplus (r_0-r_1)}\oplus  \bigoplus_{i=1}^{r_1} \cI_{n_ip},
$$
for some uniquely determined sequence of natural numbers $n_\bullet(\sF,p):=\{n_1=n_1(\sF,p)\geq \ldots \geq n_{r_1}=n_{r_1}(\sF,p)\}$, where $\cI_{n_ip}$ is the ideal sheaf of the $0$-dimensional subscheme $n_ip\subset \Xred$ in $X$. The sequence $n_\bullet(\sF,p)$ is different from the zero sequence only at the finitely many points of the support of the torsion subsheaf $\cT(\sF_{|\Xred})$ of the restriction $\sF_{|\Xred}$ and the sum (called the index of $\sF$) 
$$\iota(\sF)=\sum_{p\in X} \sum_{i=1}^{r_1} n_i(\sF,p)$$
is equal to the length of $\cT(\sF_{|\Xred})$.

\begin{ThmD}\label{T:ThmD}(see Corollaries \ref{C:Tan-lt}, \ref{C:2Tan}, \ref{C:Tan-qlf}, \ref{C:Tan-gvb})
Let $\sF\in \cM_X(r_\bullet;d_\bullet)\subset \cM_X(R,D)$ and set 
$$
\gamma(\sF)=(r_0-r_1)\iota(\sF)+\sum_{p\in X}\sum_{1\leq i\leq r_1(\sF)} (2i-1) n_i(\sF,p).
$$
\begin{enumerate}[(1)]
\item  The inclusion 
$
H^1(X,\cEnd(\sF))\subseteq T_{\sF}\cM_X(r_\bullet;d_\bullet)
$
has codimension equal to 
$$
\gamma(\sF):=(r_0-r_1)\iota(\sF)+\sum_{p\in X}\sum_{1\leq i\leq r_1(\sF)} (2i-1) n_i(\sF,p).
$$
\item  The inclusion  $T_{\sF}\cM_X(r_\bullet;d_\bullet)\subseteq T_{\sF}\cM_X(R,D)$ has codimension equal to 
$$
h^0(\Xred, \cEnd\left(\left(\frac{\sF^{(1)}}{\sN\sF}\right)^{**}\right)\otimes \sN^{-1})+\gamma(\sF),
$$
where $()^{**}$ denoted the reflexive hull. 
\end{enumerate}
In particular: 
\begin{itemize}
\item  If $\iota(\sF)=0$ (in which case $\sF$ is called quasi-locally free), then 
$$
H^1(X,\cEnd(\sF))=T_{\sF}\cM_X(r_\bullet;d_\bullet)\subseteq T_{\sF}\cM_X(R,D)
$$
and the last inclusion has codimension equal to 
$$
h^0(\Xred, \cEnd\left(\left(\frac{\sF^{(1)}}{\sN\sF}\right)^{**}\right)\otimes \sN^{-1}). 
$$
This is the case for the general element $\sF$ of any stratum $\cM_X(r_\bullet;d_\bullet)$ with $r_0>r_1$. 
 
\item  If $r_0=r_1=:r$ (in which case $\sF$ is called a generalized vector bundle of rank $r$), then the two inclusions in \eqref{E:inc-T*} have the same codimension equal to 
$$
\gamma(\sF)=\sum_{p\in X}\sum_{1\leq i\leq r_1(\sF)} (2i-1) n_i(\sF,p) \geq \iota(\sF).
$$
Furthermore, if $\sF$ is a general element of $\cM_X(r,r;d_\bullet)$ then $\gamma(\sF)=\iota(\sF)$. 
\end{itemize}
\end{ThmD}

\subsection*{Future plans}

We are planning to generalize in \cite{SV} the results of this paper to moduli of torsion-free sheaves on \emph{primitive multiple curves} of higher multiplicity $m$, i.e.  irreducible projective curves that locally looks like $mC\subset S$ for a smooth curve $C$ inside a smooth surface $S$. The basic properties of coherent sheaves on primitive multiple curves are established by Dr\'{e}zet in \cite{DR1}, \cite{DR2}, \cite{DR3}. The open subset of (semistable) generalized line bundles (i.e. sheaves that are line bundles away from a finite set of points) is studied by the first author in the unpublished draft \cite{S4}. 

\subsection*{Structure of the paper}

The paper is organized as follows. In Section \ref{Sec:FirstProp}, we collect the basic properties of coherent sheaves on ribbons that are scattered in the papers  \cite{DR1}, \cite{DR2}, \cite{DR3} of Dr\'{e}zet.
In Section \ref{Sec:stack} we introduce the stratification of the stack $\cM_X(R,D)$ by the complete type of the sheaves and study the basic properties of this stratification: non-emptiness in \S\ref{sub:nonempty}, irreducibility and dimension in \S\ref{sub:irr-dim}, specializations among strata in \S\ref{sub:speciali}.  In Section \ref{Sec:Irred}, we determine the irreducible components of the stack $\cM_X(R,D)$. In Section \ref{Sec:ss-Locus}, we investigate which strata of the stack intersect the open substacks of (semi)stable sheaves. As an example, we treat in \S\ref{sub:ss(n,1)} the semistability of sheaves of type $(n,1)$.
In Section \ref{Sec:tangent}, we study the tangent space of the stack $\cM_X(R,D)$ at a point $\sF$: we determine the endomorphism sheaf $\cEnd(\sF)$ in Theorem \ref{T:Ext0} and the higher Ext sheaves $\cExt{k}(\sF,\sF)$ for $k\geq 1$ in Theorem \ref{T:Ext1}.

\subsection*{Acknowledgements}

This paper has its origins in the PhD thesis of first author  \cite{S1} under the direction of the second author. An ancestor of this paper is the (non published) draft \cite{S3}, where the first author proved some partial results on the stratification of the moduli space of semistable torsion-free sheaves on ribbons and proposed several conjectures. In this paper, we prove those conjectures in full generality and we also extend them to the entire stack of torsion-free sheaves. 
The second author thanks Mirko Mauri for his interest in this work and for useful discussions.

 The second author is funded by the MUR Excellence Department Project MatMod@TOV awarded to the Department of Mathematics, University of Rome
Tor Vergata, CUP E83C23000330006, by the  PRIN 2022 ``Moduli Spaces and Birational Geometry''  funded by MIUR,  and he is a member of INdAM and of the Centre for Mathematics of the University of Coimbra (funded by the Portuguese Government through FCT/MCTES, DOI 10.54499/UIDB/00324/2020).

\emph{A.M.D.G.}

\section{Properties of sheaves on ribbons}\label{Sec:FirstProp}

This section collects from literature (mostly from the paper of Dr\'{e}zet \cite{DR1}, \cite{DR2}, \cite{DR3}) the properties of coherent sheaves on ribbons that we will need in the article. In this paper only coherent sheaves are considered, hence this attribute will be usually omitted in the following. Moreover, vector bundle will be used as a synonym of locally free sheaf of finite rank. 

\subsection{Ribbons}\label{sub:ribbons}

Let us begin by recalling the definition and the numerical invariants of a ribbon. We will be working over a fixed algebraically closed field $k$.

\begin{defi}\label{D:ribbon}
A \emph{ribbon} $X$ (over $k$) is a non-reduced projective $k$-scheme whose reduced subscheme $\Xred$  is a smooth connected $k$-curve and whose nilradical $\sN\subset\cO_X$, i.e. the ideal sheaf of $\Xred$ in $X$, is locally generated by a single non-zero square-zero element.
\end{defi}

 In particular, this definition implies immediately that $\sN$ can be seen as a line bundle on $\Xred$ and that it coincides with the conormal sheaf of $\Xred$ in $X$. Its dual line bundle $\sN^*$ is the normal sheaf of $\Xred$ in $X$. We will set
\begin{equation} \label{E:genus}
\begin{sis}
& \delta:=-\deg \sN, \\
& \ov g:=1-\chi(\cO_{\Xred})=\text{the genus of } \Xred,  \\
& g:=1-\chi(\cO_{X})=\text{the (arithmetic) genus of } X.
\end{sis}
\end{equation}
From the exact sequence
\[
0\to \sN \to \cO_{X}\to \cO_{\Xred}\to 0,
\]
we deduce that 
\begin{equation} \label{E:cohom}
\begin{sis}
& h^0(X,\cO_X)=1+ h^0(\Xred,\sN), \\
& h^1(X,\cO_X)=\ov g+ h^1(\Xred,\sN), \\
& g=2\ov g-1+\delta.
\end{sis}
\end{equation}




\subsection{Canonical filtrations}\label{sub:filtr}

A (coherent) sheaf  on a ribbon $X$ has two canonical filtrations, that we now recall.

\begin{defi}\label{D:CanFiltr}
Let $\sF$ be a  sheaf on $X$.
\begin{enumerate}
\item\label{D:CanFiltr:1} 
The \emph{first canonical filtration} of $\sF$ is
\[
0\subseteq\sN\sF\subseteq\sF.
\]
It is immediate to check that $\sF/(\sN\sF)=\sF|_{\Xred}$.


We define
\begin{equation}\label{E:types}
\begin{sis}
& r_{\bullet}(\sF):=(r_0(\sF),r_1(\sF)):=(\rk{\sF|_{\Xred}},\rk{\sN\sF})=\text{\emph{complete rank} of } \sF, \\
& d_{\bullet}(\sF):=(d_0(\sF),d_1(\sF)):=(\deg(\sF|_{\Xred}),\deg(\sN\sF))=\text{\emph{complete degree} of } \sF, \\
& (r_\bullet(\sF);d_\bullet(\sF)):=(r_0(\sF),r_1(\sF);d_0(\sF),d_1(\sF))= \text{\emph{complete type} of } \sF.
\end{sis}
\end{equation}


\item\label{D:CanFiltr:2}

The \emph{second canonical filtration} of $\sF$ is defined as 
\[
0\subseteq\sF^{(1)}\subseteq\sF,
\]
where $\sF^{(1)}$ is the subsheaf of $\sF$ annihilated by $\sN$.

\end{enumerate}
\end{defi}

The complete type of $\sF$ can be determined by looking at the structure of the generic stalk $\sF_{\eta}$, where $\eta$ is the generic point of $X$. 

\begin{rmk}\label{R:type}
If $\eta$ is the generic point of $X$, then we have that 
$$\sF_\eta\cong\cO_{\Xred,\eta}^{\oplus a(\sF)}\oplus\cO_{X,\eta}^{\oplus b(\sF)},$$
for some unique non-negative integers $(a(\sF),b(\sF))$, called the \emph{type} of $\sF$.
The complete rank of $\sF$ is determined by the type of $\sF$ (and vice-versa) by virtue of the following formulas: 
$$r_0(\sF)=a(\sF)+b(\sF) \quad \text{ and } \quad r_1(\sF)=b(\sF).$$

\end{rmk}

The following fact collects the main properties of the two canonical filtrations and their mutual relations.

\begin{fact}\label{F:FiltrCan}
Let $\sF$ be a sheaf on  $X$.
\begin{enumerate}
\item\label{F:FiltrCan:1} (see \cite[\S 3.1.4]{DR2}, \cite[\S 3.2.3]{DR4}) There is a canonical isomorphism
\[
\sN\sF=(\sF/\sF^{(1)})\otimes\sN.
\]
Moreover there is a canonical exact sequence on $\Xred$ (called the \emph{canonical exact sequence of} $\sF$):
\begin{equation}\label{Eq:ExSeq1}
0\to\sN\sF\to\sF^{(1)}\to\sF|_{\Xred}\to\sN\sF\otimes\sN^{-1}\to 0.
\end{equation}

\item\label{F:FiltrCan:2} (see \cite[\S 3.2.4]{DR4})
Let $\cF$ be a subsheaf of $\sF$. Then:
\begin{enumerate}[(a)]
\item $\cF$ is defined on $\Xred$ if and only if $\cF\subseteq \sF^{(1)}$;
\item $\sF/\cF$ is defined on $\Xred$ if and only if $\sN\sF\subseteq\cF$.
\end{enumerate}
Moreover,  in the latter case, there is a canonical morphism $\sF/\cF\otimes\sN\to\cF$, which is surjective if and only if $\cF=\sN\sF$, while it is injective if and only if $\cF=\sF^{(1)}$.

\item\label{F:FiltrCan:3} (see \cite[\S 3.4.2, \S 3.4.3]{DR2}, \cite[\S 3.4]{DR4}) Let $\cF$ be a sheaf on $\Xred$ and let $\cE$ be a vector bundle on it. Then there exists the following canonical exact sequence:
\begin{equation}\label{E:sucExt}
\phantom{00000.}0\to \Ext{1}_{\cO_{\Xred}}(\cE,\cF)\to \Ext{1}_{\cO_{X}}(\cE,\cF)\overset{\pi}{\longrightarrow} \Hom(\cE\otimes\sN,\cF)\to 0.  
\end{equation}
By \ref{F:FiltrCan:2}, if $\sF$ is a sheaf on $X$ sitting in a short exact sequence $0\to\cF\to\sF\to\cE\to 0$ represented by $\sigma\in\Ext{1}_{\cO_{X}}(\cE,\cF)$, then $\cF=\sN\sF$ if and only if $\pi(\sigma)$ is surjective, while $\cF=\sF^{(1)}$ if and only if $\pi(\sigma)$ is injective.

\item\label{F:FiltrCan:4}(see \cite[Prop. 3.5]{DR2}) Let $\varphi:\sF\to\sG$ be a morphism of sheaves on $X$. Then:
\begin{enumerate}

\item\label{F:FiltrCan:4:1} $\varphi$ is surjective if and only if the induced morphism $\varphi|_{\Xred}:\sF|_{\Xred}\to \sG|_{\Xred}$ is surjective. If this is case,  the induced morphism $\sN\sF\to \sN\sG$ is also surjective.

\item\label{F:FiltrCan:4:2} $\varphi$ is injective if and only if the induced morphism $\sF^{(1)}\to \sG^{(1)}$ is injective. If this is the case, the induced morphism $\sF/\sF^{(1)}\to \sG/\sG^{(1)}$ is also injective.
\end{enumerate}
\item\label{F:FiltrCan:5}(see \cite[Prop. 7.3.1]{DR1}) $r_0(\sF)$ is upper semicontinuous, while $r_1(\sF)$ is lower semicontinuous.
\end{enumerate}
\end{fact} 

\begin{rmk}\label{R:inv-F1}
Thanks to Fact \ref{F:FiltrCan}\ref{F:FiltrCan:1}, the complete type of a sheaf $\sF$ can be characterized also in terms of the second canonical filtration of $\sF$:
\begin{equation}\label{E:typebis}
\begin{sis}
& (r_0(\sF),r_1(\sF))=(\rk{\sF^{(1)}}, \rk{\sF/\sF^{(1)}}), \\
& (d_0(\sF),d_1(\sF))=(\deg(\sF^{(1)})+r_1\delta,\deg(\sF/\sF^{(1)})-r_1\delta). 
\end{sis}
\end{equation}
\end{rmk}

\subsection{Generalized rank and degree}\label{sub:R-D}

Now we introduce two fundamental invariants of a (coherent) sheaf on a ribbon $X$: the generalized rank and the generalized degree. 
\begin{defi}\label{D:GenRkDeg}
Let $\sF$ be a sheaf on $X$. 
\begin{enumerate}
\item  The \emph{generalized rank} $\Rk{\sF}$ of $\sF$ is its generic length, i.e. the length of its generic stalk $\sF_{\eta}$ as an $\cO_{X,\eta}$-module (here and throughout the paper $\eta$ denotes the generic point of $X$).
\item  The \emph{generalized degree} $\Deg{\sF}$ of $\sF$ is 
$$\Deg{\sF}=d_0(\sF)+d_1(\sF).$$

\end{enumerate}
\end{defi}

\begin{rmk}\label{R:GenRkDeg} \noindent
\begin{enumerate}
\item\label{R:GenRkDeg:2} If $\cF$ is defined on $\Xred$, meaning that it is the direct image on $X$ of a sheaf on $\Xred$, then
$$
\Rk{\cF}=\rk{\cF} \quad \text{ and } \quad \Deg{\cF}=\deg(\cF),
$$
where $\rk{\cF}$ and $\deg(\cF)$ are the usual rank and degree of a sheaf on $\Xred$. 
 



\item\label{R:GenRkDeg:1} 
If we have an exact sequence
$$0\to \cF\to \sF\to \cE\to 0,$$
with $\cF$ and $\cE$ sheaves on $\Xred$, then it holds that 
$$
\Rk{\sF}=\rk{\cF}+\rk{\cE} \quad \text{ and } \quad \Deg{\sF}=\deg(\cF)+\deg(\cE).
$$
In particular, the generalized rank of a sheaf is determined by its complete rank:
$$
\Rk{\sF}=r_0(\sF)+r_1(\sF)=a(\sF)+2b(\sF).
$$
\end{enumerate}
\end{rmk}

The following fact collects the basic properties of the generalized rank and degree. 
\begin{fact}\label{F:GenRkDeg}
\noindent
\begin{enumerate}
\item\label{F:GenRkDeg:1} (see \cite[\S 4.2.2]{DR1})
Let $\cO_{X}(1)$ be a very ample line bundle on $X$, let $\cO_{\Xred}(1)$ be its restriction to $\Xred$ and let $d=\deg{\cO_{\Xred}(1)}$. If $\sF$ is a sheaf on $X$, its Hilbert polynomial with respect of $\cO_{X}(1)$ is
\begin{equation}\label{Eq:HilbPol}
P_{\sF}(T)=\Deg{\sF}+\Rk{\sF}(1-\ov g)+\Rk{\sF}dT. 
\end{equation}

\item\label{F:GenRkDeg:2}
(see \cite[Cor. 4.3.2]{DR1}) The generalized rank and degree are additive: if 
$$0\to \sF'\to \sF\to \sF''\to 0$$ 
is an exact sequence of sheaves on $X$, then 
$$\Rk{\sF}=\Rk{\sF'}+\Rk{\sF''} \quad \text{ and } \quad \Deg{\sF}=\Deg{\sF'}+\Deg{\sF''}.$$

\item\label{F:GenRkDeg:3} (see \cite[Prop. 4.3.3]{DR1})
 The generalized rank and degree of sheaves on $X$ are invariant by deformation.

\item \label{F:GenRkDeg:4} (see \cite[Thm. 4.2.1]{DR1})
The following Riemann-Roch formula holds
$$
\chi(\sF)=\Deg{\sF}+\Rk{\sF}(1-\ov g).
$$

\end{enumerate}
\end{fact}

\subsection{Purity and duality}\label{sub:pure-dual}

We now introduce the notions of pure, torsion-free and reflexive sheaves. 

\begin{defi}\label{D:pure}
Let $\sF$ be a sheaf on $X$.
\begin{enumerate}[(i)]
\item 
$\sF$ is said to be  \emph{pure} of dimension $d(=0 \text{ or } 1)$ if the dimension of any non-zero subsheaf (i.e. the dimension of its support) is equal to $d$.
\item 
 $\sF$ is said to be \emph{torsion-free} if every non-zerodivisor of $\cO_X(U)$ is a non-zerodivisor also of $\sF(U)$ for any open $U\subseteq X$ (recall that $f\in H^0(U,\cO_X)$ is a non-zerodivisor of $\sF$ if the multiplication map $f \cdot \_ :\sF|_U\to\sF|_U$ is injective).
\item 
$\sF$ is \emph{reflexive} if the canonical morphism $\sF\to\sF^{\vee\vee}$ is an isomorphism, where $\sF^\vee:=\sHom{\sF,\cO_{X}}$ is the \emph{dual} of $\sF$. 

\end{enumerate}
\end{defi}
Note that $\sF$ is pure of dimension $0$ if and only if it is a torsion sheaf, i.e. it has finite support. 

\begin{rmk}\label{Rmk:Dual} \noindent
If $\cF$ is a sheaf on $\Xred$, then its dual $\cF^{\vee}$ on $X$ is related to its dual $\cF^{*}:=\sHom{\cF,\cO_{\Xred}}$ on $\Xred$ by a canonical isomorphism (see \cite[Lemme 4.1]{DR2}):
$$\cF^{\vee}\simeq\cF^{*}\otimes\sN.$$ 
Hence, $\cF$ is  reflexive on $X$ if and only if it is reflexive on $\Xred$, i.e. it is a vector bundle on $\Xred$. 
\end{rmk}

The relationship between the above  notions is described in the following fact.

\begin{fact}\label{F:Dual} Let $\sF$ be a sheaf on $X$. Then: 
\begin{enumerate}

\item\label{F:Dual:2} (see \cite[Cor. 4.6]{DR2}) It holds that $\cExt{i}(\sF,\cO_{X})=0$ for any $i\ge 2$. 

\item\label{F:Dual:1} (see \cite[Lemma 2.2]{CK} and \cite[Prop. 3.8 and Thm. 4.4]{DR2}) The following are equivalent:
\begin{enumerate}
\item $\sF$ is pure of dimension one;

\item $\sF$ is torsion-free;

\item $\sF$ is reflexive;

\item $\sF^{(1)}$ is a vector bundle on $\Xred$;

\item it holds that $\cExt{1}(\sF,\cO_{X})=0$.
\end{enumerate}
Moreover, if the above conditions hold, then $\sF/\sF^{(1)}$ and $\sN\sF$ are vector bundles on $\Xred$. 
 

\item\label{F:Dual:4} (see \cite[Prop. 4.4.1]{DR3}) We have that 
$$
\Rk{\sF^\vee}=\Rk{\sF} \quad \text{ and } \quad \Deg{\sF^\vee}=-\Deg{\sF}+\Rk{\sF}\deg{\sN}+h^0(\cT(\sF)),$$ 
where $\cT(\sF)$ is the torsion subsheaf of $\sF$, i.e. its greatest subsheaf with finite support.

\item\label{F:Dual:4bis} (see \cite[Cor. 4.5]{DR2}) If $0\to\sE\to\sF\to\sG\to 0$ be a short exact sequence of sheaves on $X$ with $\sG$ torsion-free, then also the dual sequence $0\to\sG^\vee\to\sF^\vee\to\sE^\vee\to 0$ is exact.

\item\label{F:Dual:5} Assume that $\sF$ is torsion-free. 
\begin{enumerate}
\item\label{F:Dual:5:1} (see \cite[Prop. 4.3.1(i)]{DR3}) There is a canonical isomorphism 
$$\cT(\sF^{\vee}|_{\Xred})\cong \cExt{1}(\cT(\sF|_{\Xred}),\cO_X)\otimes\sN,$$ 
where $\cT(\sF^{\vee}|_{\Xred})$ and $\cT(\sF|_{\Xred})$ are, respectively, the torsion subsheaves of $\sF^{\vee}|_{\Xred}$ and $\sF|_{\Xred}$. 
\item\label{F:Dual:5:2} (see \cite[Prop. 4.2]{DR2} and \cite[Prop. 4.3.1(ii)]{DR3}) 
The dual of the  exact sequence 
$$
0\to \sF^{(1)} \to \sF\to \sN\sF\otimes \sN^{-1}\to 0
$$
is equal to 
$$0\to (\sN\sF^\vee)^{\rm sat} \to \sF^{\vee} \to  (\sF^{\vee}_{|\Xred})^{\vee\vee} \to 0$$
where $(\sN\sF^\vee)^{\rm sat}$ is the saturation of $\sN\sF^\vee$ in  $\sF^{\vee}$.  
\end{enumerate}

\item\label{F:Dual:6} (see \cite[Thm. 4.7]{DR2}) Assume that $\sF$ is torsion-free. Then there are functorial Serre-duality isomorphisms 
$$
H^i(X,\sF)\cong H^{1-i}(X,\sF^\vee\otimes \omega_X) \quad \text{ for  } i=0,1. 
$$
\end{enumerate}
\end{fact}

\subsection{Quasi-locally free sheaves and sheaves of rigid type}\label{sub:qlf-rig}

In this subsection, we introduce two special classes of torsion-free sheaves, namely quasi-locally free sheaves and sheaves of rigid type. 

We begin by a structure result for the stalks of a torsion-free sheaf on $X$.

\begin{fact}\label{F:structure}(see \cite[Cor. 6.5.5]{DR1})
Let $\sF$ be a torsion free sheaf of type $(a,b)=(a(\sF), b(\sF))$. Then  for any closed point $p\in X$ there exists a unique sequence of natural numbers $n_\bullet(\sF,p):=\{n_1=n_1(\sF,p)\geq \ldots \geq n_b=n_b(\sF,p)\}$ such that 
$$
\sF_p= \cO_{\Xred,p}^{\oplus a}\oplus  \bigoplus_{i=1}^b \cI_{n_ip},
$$
where $\cI_{n_ip}$ is the ideal sheaf of the $0$-dimensional subscheme $n_ip\subset \Xred$ in $X$. 
\end{fact}
This follows from \cite[Cor. 6.5.5]{DR1} together with the fact that $\cI_{n_ip}$ has type $(1,0)$.

\begin{defi}\label{Def:qll} (see \cite[\S 5.1]{DR1}.)
Let $\sF$ be a sheaf on $X$ of type $(a,b)=(a(\sF), b(\sF))$.  We say that:
\begin{enumerate}
\item $\sF$ is \emph{quasi-locally free in a closed point} $p\in X$ if 
$$\sF_p\cong \cO_{\Xred,p}^{\oplus a}\oplus \cO_{X,p}^{\oplus b},$$
i.e. if $n_i(\sF,p)=0$ for any $1\leq i \leq b$.  
\item $\sF$ is \emph{quasi-locally free} if it is quasi-locally free in any closed point of $X$. 
\end{enumerate}
\end{defi}


The following fact contains the main properties of quasi-locally free sheaves.

\begin{fact}\label{Fact:qll} Let $\sF$ be a sheaf on $X$.
\begin{enumerate}

\item\label{Fact:qll:1}(see \cite[Thm. 3.9 and Cor. 3.10]{DR2})
Let $p$ be a closed point of $X$. The following assertions are equivalent:
\begin{enumerate}
\item $\sF$ is quasi-locally  free (resp. quasi-locally  free in $p$);

\item $\sN\sF$ and $\sF|_{\Xred}$ are locally free on $\Xred$ (resp. are locally free in $p$);

\item $\sF^{(1)}/\sN\sF$ and $\sF|_{\Xred}$ are locally free on $\Xred$ (resp. are locally free in $p$);
\item $\sF$ and $\sF|_{\Xred}$ are torsion-free (resp. are torsion-free in $p$).
\end{enumerate}

\item\label{Fact:qll:2}(see \cite[\S 6.5.2]{DR1}) If $\sF$ is quasi-locally free, then the dual of the canonical sequence of $\sF$ is the canonical sequence of $\sF^{\vee}$.
In particular, we have that 
$$
(\sN\sF)^{\vee}=\sN\sF^{\vee} \otimes\sN^{-1}, \quad (\sF^{(1)})^{\vee}=(\sF^{\vee})_{|\Xred}, \quad (\sF_{|\Xred})^{\vee}=(\sF^\vee)^{(1)}. 
$$

\item\label{Fact:qll:3}(see \cite[Thm. 5.1.6]{DR1}) $\sF$ is generically quasi-locally free, i.e. there exists an open $\emptyset\neq U\subseteq X$ such that $\sF$ is quasi-locally free in each point of $U$.

\item\label{Fact:qll:4}(see \cite[\S 5.1]{DR2}) If $\sF$ is quasi-locally free, then there exists a vector bundle $\bE$ on $X$ and a surjective morphism 
$$
\phi:\bE\twoheadrightarrow \sF
$$ 
such that $\phi_{|\Xred}:\bE_{|\Xred}\to \sF_{|\Xred}$ is an isomorphism. In particular, $\bE$ has generalized rank and degree equal to 
$$
\begin{sis}
& \Rk{\bE}=2r_0(\sF),\\
& \Deg{\bE}=2d_0(\sF)-\delta r_0(\sF),
\end{sis}
$$ 
and $\ker{\phi}$ is a vector bundle on $\Xred$ of rank and degree equal to 
$$
\begin{sis}
& \rk{\ker{\phi}}=r_0(\sF)-r_1(\sF),\\
& \deg(\ker{\phi})=d_0(\sF)-d_1(\sF)-\delta r_0(\sF).
\end{sis}
$$ 

\end{enumerate}
\end{fact}

The failure of a torsion-free sheaf to be quasi-locally free is measured by the local and global index.

\begin{defi}\label{Def:index} (see \cite[\S 6.3.7]{DR1},  \cite[Def. 2.7]{CK})
Let $\sF$ be a torsion-free sheaf on $X$ and let $\cT:=\cT(\sF|_{\Xred})$ be the torsion subsheaf of $\sF|_{\Xred}$. 
\begin{enumerate}
\item For any closed point $p\in X$, the \emph{local index} of $\sF$ at $p$ is
$$\iota_p(\sF):= \text{length of } \cT_p \text{ as an } \cO_{\Xred,p}-\text{module}.$$ 
\item The \emph{index} of $\sF$ is 
$$\iota(\sF):=h^0(\cT)=\sum_{p\in X} \iota_p(\sF).$$ 
\end{enumerate}
\end{defi}

\begin{rmk}\label{Rmk:Def:index}
Let $\sF$ be a torsion-free sheaf. For any closed point $p\in X$, we have that 
$$
\iota_p(\sF)=\sum_{i=1}^{b(\sF)}n_i(\sF, p).
$$
This follows from the fact that $\cT((\cI_{n_ip})_{|\Xred})=\cO_{\Xred,p}/{\mathfrak m}_p^{n_i}=\cO_{n_ip}$ (see \cite[Exa. 6.3.10]{DR1}).
In particular,  $\iota(\sF)=0$ if and only if $\sF$ is quasi-locally free.
\end{rmk}

The following fact relates torsion-free sheaves with positive index to quasi-locally free ones.

\begin{fact}\label{Fact:Nonqll}(see \cite[\S 6.3, \S 6.4]{DR1})
Let $\sF$ be a torsion-free sheaf on $X$ and let $\cT$ be the torsion subsheaf of $\sF|_{\Xred}$.
There exist two quasi-locally free sheaves $\sE$ and $\sG$ on $X$ (not necessarily unique) 
and two exact sequences:
\begin{align*}
0\longrightarrow\sE\longrightarrow\sF\longrightarrow\cT\longrightarrow 0,\\
0\longrightarrow\sF\longrightarrow\sG\longrightarrow\cT\longrightarrow 0,
\end{align*}
such that 
$\sN\sE=\sN\sF$ and $\sG^{(1)}=\sF^{(1)}$.
\end{fact}

Now we introduce sheaves of rigid type.

\begin{defi}\label{D:rigid}
A torsion-free sheaf $\sF$ on $X$ is said to be \emph{of rigid type} if:
\begin{enumerate}
\item either $\Rk{\sF}$ is even and $(a(\sF),b(\sF))=\left(0,\frac{\Rk{\sF}}{2}\right)$ (or equivalently $r_0(\sF)=r_1(\sF)=\frac{\Rk{\sF}}{2}$);
\item or $\Rk{\sF}$ is odd and $(a(\sF),b(\sF))=\left(1,\frac{\Rk{\sF}-1}{2}\right)$ (or equivalently $r_0(\sF)=\frac{\Rk{\sF}+1}{2}>\frac{\Rk{\sF}-1}{2}=r_1(\sF)$).
\end{enumerate}
Torsion-free sheaves of rigid type and even generalized rank are called \emph{generalized vector bundles}.  The rank of a generalized vector bundle $\sF$ is the number $r_0(\sF)=r_1(\sF)$. 
\end{defi}


\begin{rmk}\label{R:rigid}
\noindent 
\begin{enumerate}
\item \label{R:rigid1} The condition of being of rigid type is open in families by Fact \ref{F:FiltrCan}\ref{F:FiltrCan:5}.
\item \label{R:rigid3} $\sF$ is a generalized vector bundle of index zero (or equivalently quasi-locally free) if and only if  $\sF$ is a vector bundle (see Fact \ref{Fact:qll}\ref{Fact:qll:4}).
\end{enumerate}
\end{rmk}

\section{The stack of torsion-free sheaves}\label{Sec:stack}

The aim of this section is to study the stack of torsion-free sheaves on a ribbon $X$, which we now introduce. 

\begin{defi}\label{D:stack}
Let $X$ be a ribbon.  
\begin{enumerate}
\item We denote by $\cM_X$ the \emph{stack of torsion-free sheaves} on $X$, i.e. the (algebraic) stack whose fiber over a $k$-scheme $S$ is the groupoid of 
coherent sheaves $\mathfrak{F}$ on $X\times_{k} S$ that are flat over $S$ and whose restriction to the geometric fibers of $X\times_{k} S\to S$  are torsion-free sheaves on $X$.  

\item Since the generalized degree and degree are invariant under deformations (see Fact \ref{F:GenRkDeg}\ref{F:GenRkDeg:3}), we have disjoint decompositions into open and closed substacks
\begin{equation}\label{E:decMR}
\cM_X=\coprod_{R\in \bN_{>0}} \cM_X(R)=\coprod_{(R,D)\in \bN_{>0}\times \bZ} \cM_X(R,D),
\end{equation}
where $\cM_X(R)$ (resp. $\cM_X(R, D)$) the stack parametrizing torsion-free sheaves on $X$ of generalized rank $R$ (resp. and generalized degree $D$). 
\end{enumerate}
\end{defi}

We now want to define a stratification of $\cM_X(R,D)$ according to the complete type of the sheaves. For that purpose, we introduce the following 

\begin{defi}\label{D:seq}
Fix a pair $(R, D)\in \bN_{>0} \times \bZ$.
\begin{enumerate}
\item A \emph{sequence of type} $(R,D)$ is a sequence  $(r_\bullet;d_\bullet)=(r_0,r_1;d_0,d_1)\in \bN^2\times \bZ^2$ such that $r_0+r_1=R$ and $d_0+d_1=D$.
The collection of all sequences of type $(R,D)$ is denoted by $\cS(R,D)$.
\item A sequence $(r_\bullet;d_\bullet)\in \cS(R,D)$ is called \emph{admissible} if it satisfies the following conditions
$$
\begin{sis} 
&r_0\geq r_1, \\
&r_1=0\Rightarrow d_1=0,\\
& r_0=r_1 \Rightarrow d_0\geq d_1+r_1\delta.
\end{sis}
$$
The collection of all admissible sequences of type $(R,D)$ is denoted by $\cS_{adm}(R,D)$. 
\item Given $(r_\bullet;d_\bullet), (\wt r_\bullet;\wt d_\bullet)\in \cS(R,D)$, we say that 
$$(r_\bullet;d_\bullet)\geq (\wt r_\bullet;\wt d_\bullet) 
\stackrel{def}{\Longleftrightarrow} \text{either } r_1> \wt r_1 \: \text{ or } r_1=\wt r_1 \: \text{ and } \: d_1\geq \wt d_1.$$
\item A sequence $(r_\bullet;d_\bullet)\in \cS(R,D)$ is called \emph{of rigid type} if either $r_0=r_1$ (which forces $R$ even) or $r_0=r_1+1$ (which forces $R$ odd). 
The collection of all rigid  sequences of type $(R,D)$ is denoted by $\cS_{rig}(R,D)$. 
\end{enumerate}
\end{defi}

We can now introduce the desired stratification of $\cM_X(R,D)$ by complete type. 
 
\begin{defi}\label{D:strata}
Given a sequence $(r_\bullet;d_\bullet)\in \cS(R,D)$, we define the following locus 
$$
\cM_X(r_{\bullet};d_{\bullet}):=\{\cF\in \cM_X(R,D)\: : \: (r_\bullet(\cF);d_\bullet(\cF))=(r_{\bullet};d_{\bullet})\}\subseteq \cM_X(R,D).
$$ 
\end{defi}

\begin{rmk}\label{R:strata}
\noindent
\begin{enumerate}
\item \label{R:strata1} The strata $\cM_X((r_0,0);d_{\bullet}))$ parametrize torsion-free sheaves  (and hence locally free)  on $\Xred$ of rank $r_0$. Hence it is non-empty if and only if $d_{\bullet}=(d_0,d_1)=(d_0,0)$ in which case $d_0$ is equal to the degree of  the sheaves on $\Xred$. 
\item \label{R:strata2} The strata $\cM_X((r,r);d_{\bullet}))$ parametrize generalized vector bundles on $\Xred$ of  rank $r$. 
\item  \label{R:strata3} The stratum $\cM_X(r_{\bullet};d_{\bullet})$ parametrizes rigid torsion-free sheaves on $X$ if and only if $(r_{\bullet};d_{\bullet})$ is of rigid type. 
\end{enumerate} 
\end{rmk}

In the next subsections, we investigate the properties of the stratification of $\cM_X(R,D)$ by complete types.

\subsection{Non-emptiness}\label{sub:nonempty}

In this subsection, we determine which strata are non-empty.

\begin{thm}\label{T:non-empty}
Let $X$ be a ribbon and fix any $(R, D)\in \bN_{>0} \times \bZ$. For any  $(r_\bullet;d_\bullet)\in \cS(R,D)$, we have that 
\begin{equation}\label{E:adm-seq}
\cM_X(r_{\bullet};d_{\bullet})\neq \emptyset \Longleftrightarrow (r_\bullet;d_\bullet)\in \cS_{adm}(R,D).
\end{equation}
\end{thm}
\begin{proof}
First of all, let us prove that if $\cM_X(r_{\bullet};d_{\bullet})\neq \emptyset $ then $(r_\bullet, d_\bullet)$ is admissible. Indeed, the fact that $r_0\geq r_1$ follows from Remark \ref{R:type}, while the implication $r_1=0\Rightarrow d_1=0$ follows from Remark \ref{R:strata}\ref{R:strata1}. 
Finally, if $\sF$ is a torsion-free sheaf of complete rank  $(r_0,r_0)$, then $\sF$ is a generalized vector bundle of rank $r_0$, and the inequality $d_0(\sF)\geq d_1(\sF)+r\delta$ follows from Lemma \ref{L:genVB}.

Assume now that $(r_\bullet, d_\bullet)$ is admissible and let us show that $\cM_X(r_{\bullet};d_{\bullet})\neq \emptyset $. If $r_1=0$, then the non-emptiness of $\cM_X(r_0,0;d_0, 0)$ follows from 
Remark \ref{R:strata}\ref{R:strata1}. If $r_0>r_1>0$ then the non-emptiness of $\cM_X(r_{\bullet};d_{\bullet})$ follows from Proposition \ref{P:ex-qufree} applied to an exact sequence \eqref{Eq:ExSeq*}
of vector bundles on $\Xred$ such that 
$$
\deg(\cF)=d_1,  \deg(\cE)=d_0-r_1\delta,  \deg(\cG)=d_0.
$$
 If $r_0=r_1>0$ then the non-emptiness of $\cM_X(r_{\bullet};d_{\bullet})$ follows from Proposition \ref{P:ex-qufree} applied to an exact sequence  \eqref{Eq:ExSeq*}
such that $\cF$ is a vector bundle of degree $d_1$, $\cE$ is vector bundle of degree $d_0-r_1\delta$, $\cG$ is a coherent sheaf with torsion subsheaf $\cT(\cG)=\Im(e)$ of degree 
$d_0-d_1-r_1\delta\geq 0$ and locally free quotient equal to $\cF\otimes \sN^{-1}$. 
\end{proof}

\begin{lemma}\label{L:genVB}
If $\sF$ is a generalized vector bundle of  rank $r$, generalized degree $D$ and index $\iota(\sF)$, then the complete degree of $\sF$ is given by 
\begin{equation}\label{E:b->d}
d_0(\sF)=\frac{D+\iota(\sF)+r\delta}{2} \quad \text{ and } \quad d_1(\sF)=\frac{D-\iota(\sF)-r\delta}{2}.
\end{equation}
In particular, we have that 
\begin{equation}\label{E:d->b}
\iota(\sF)=d_0(\sF)-d_1(\sF)-r\delta. 
\end{equation}
\end{lemma}
\begin{proof}
We have a surjection 
$$\sF|_{\Xred}\twoheadrightarrow \sF/\sF^{(1)}\cong \sN\sF\otimes \sN^{-1}$$ 
between sheaves on $\Xred$ of the same rank $r$, and such that $\sF/\sF^{(1)}$ is a vector bundle.
This implies that we have an exact sequence 
$$0\to \cT(\sF|_{\Xred})\to \sF|_{\Xred}\to \sF/\sF^{(1)}\cong \sN\sF\otimes \sN^{-1}\to 0,$$
which implies that 
$$
d_0(\sF)=\iota(\sF)+d_1(\sF)-r\deg \sN=\iota(\sF)+d_1(\sF)+r\delta.
$$
We conclude using that $D=d_0(\sF)+d_1(\sF)$.
\end{proof}

\begin{ex}\label{Ex:genlb}
Any \emph{generalized line bundle}, i.e. a generalized vector bundle of   rank $1$, is of the form $\cI_D\otimes \cL$ for a unique $0$-dimensional subscheme $D\subset \Xred$ and a (not necessarily unique) line bundle $\cL$ on $X$ (see \cite[Prop. 8.2.1]{DR1}).
The restriction of  $\cI_D\otimes \cL$ to $\Xred$ is given by (see \cite[Ex. 6.3.10]{DR1})
$$(\cI_D\otimes \cL)_{|\Xred}=\cL_{|\Xred}(-D)\oplus \cO_D.$$ 
Hence, using Fact \ref{F:FiltrCan}\ref{F:FiltrCan:1}, we deduce that the canonical exact sequence of $\cI_D\otimes \cL$ is equal to 
$$
0\to \cL_{|\Xred}(-D)\otimes \sN\to \cL_{|\Xred}\otimes \sN \to \cL_{|\Xred}(-D)\oplus \cO_D\to  \cL_{|\Xred}(-D)\to 0
$$
where the middle homomorphism  factors through $\cO_D$. It follows that 
$$
\begin{sis}
& b(\cI_D\otimes \cL)=l(D), \\
& d_0(\cI_D\otimes \cL)=\deg \cL_{|\Xred}, \\
& d_1(\cI_D\otimes \cL)=\deg \cL_{|\Xred}-l(D)-\delta,\\
& \Deg{\cI_D\otimes \cL}=2\deg \cL_{|\Xred}-l(D)-\delta.
\end{sis}
$$
where $l(D)$ is the length of $D$. 
\end{ex}

\begin{prop}\label{P:ex-qufree}
Let $(r_0,r_1)$ be a pair of positive integers with $r_0\geq r_1$ and let 
\begin{equation}\label{Eq:ExSeq*}
0\to \cF \overset{f}{\longrightarrow} \cE\overset{e}{\longrightarrow}\cG\overset{g}{\longrightarrow} \cF\otimes\sN^{-1}\to 0
\end{equation}
 be an exact sequence of coherent sheaves on $\Xred$, with $\rk{\cF}=r_1$, $\rk{\cG}=\rk{\cE}=r_0$ and  $\cF$ locally free. Then there exists a coherent sheaf $\sF$ on $X$ such that its associated canonical exact sequence \eqref{Eq:ExSeq1} is isomorphic to \eqref{Eq:ExSeq*}. 
\end{prop}
\begin{proof}
Using Fact \ref{F:FiltrCan}\ref{F:FiltrCan:3} and the surjection $\cE\twoheadrightarrow \cK:=\ker{g}=\im{e}$, we get a commutative diagram 
 \[
\begin{tikzcd}
0 \ar[r] & \Ext{1}_{\cO_{\Xred}}(\cF\otimes \sN^{-1},\cE) \ar[d, twoheadrightarrow, "\wt{p}"] \ar[r] &\Ext{1}_{\cO_X}(\cF\otimes \sN^{-1},\cE) \ar[r, "\pi"] \ar[d, "p"] &\Hom(\cF,\cE) \ar[d, "\wh{p}"] \ar[r] & 0 \\
0 \ar[r] & \Ext{1}_{\cO_{\Xred}}(\cF\otimes \sN^{-1},\cK) \ar[r, "\iota"] &\Ext{1}_{\cO_{X}}(\cF\otimes \sN^{-1},\cK)\ar[r]  &\Hom(\cF,\cK)\ar[r] & 0
\end{tikzcd}
\] 
whose rows are exact by Fact \ref{F:FiltrCan}\ref{F:FiltrCan:3} (using that $\cF$ is locally free) and such that $\wt{p}$ is surjective because its cokernel lives inside $\Ext{2}_{\cO_{\Xred}}(\cF\otimes \sN^{-1},\cF)=H^2(\Xred,\sN)=0$ (using again that  $\cF$ is locally free). 

By the above diagram and using the fact that $\wh{p}(f)=0$ because $\cK=\coker{f}$, 
there exists a $\sF$ be a sheaf over $X$ corresponding to an element $\sigma_{\sF}\in\Ext{1}_{\cO_X}(\cG,\cF)$ such that 
\begin{enumerate}[(a)]
\item $\pi(\sigma_{\sF})=f$;
\item $p(\sigma_{\sF})=\iota([G])$, where $[G]=[0\to \cK\to \cG\to \cF\otimes \sN^{-1}\to 0]$. 
\end{enumerate}

Property (a), together with the injectivity of $f$ and  Fact \ref{F:FiltrCan}\ref{F:FiltrCan:3}, implies that we have canonical identifications 
$$ 
\sF^{(1)}=\cE \text{ and } \sF/\sF^{(1)}=\cF\otimes \sN^{-1},
$$
and that, moreover, the injective map $f:\cF\to \cE$ is identified with the injective map $\sN\sF\hookrightarrow \sF^{(1)}$. It follows that we have also a canonical 
identification
$$\cK=\coker{\sN\sF\hookrightarrow \sF^{(1)}}=\ker{\sF_{|\Xred}\twoheadrightarrow \sF/\sF^{(1)}}.$$ 

Property (b) now implies that we have a canonical identification of extensions
$$
[0\to \cK\to \cG\to \cF\otimes \sN^{-1}\to 0]= [0\to \cK\to \sF_{|\Xred}\rightarrow \sF/\sF^{(1)}=\cF\otimes \sN^{-1} \to 0].
$$
This shows that the canonical exact sequence \eqref{Eq:ExSeq1} associated to $\sF$ is isomorphic to $\eqref{Eq:ExSeq*}$, and we are done. 
\end{proof}

\begin{rmk}\label{R:exis-qlf}
It follows from the proof of Theorem \ref{T:non-empty} together with Fact \ref{Fact:qll}\ref{Fact:qll:1} that if $r_0>r_1>0$ then $\cM_X(r_0,r_1;d_\bullet)$ contains quasi-locally free sheaves. 
\end{rmk}

\subsection{Irreducibility and dimension}\label{sub:irr-dim}

In this subsection, we prove that the strata are smooth and irreducible and we determine their dimensions.

\begin{thm}\label{T:irr-dim}
Let $X$ be a ribbon and fix any $(R, D)\in \bN_{>0} \times \bZ$. 
If $(r_\bullet;d_\bullet)\in \cS_{adm}(R,D)$ then $\cM_X(r_{\bullet};d_{\bullet})$ is a locally closed substack  which, endowed with its reduced stack structure, is smooth and irreducible of dimension equal to 
$$
\dim \cM_X(r_{\bullet};d_{\bullet})=(r_0^2+r_1^2)(\ov g -1)+r_0r_1\delta=R^2(\ov g-1)+r_0r_1[\delta -2(\ov g-1)].
$$
\end{thm}

The above Theorem generalizes some known special cases:  \cite[Prop. G]{DR2} for quasi-locally free sheaves of complete rank $(r_0,r_1)=(a+1,a)$ for $a\geq 1$ (extending the case $a=1$ treated in \cite[Prop. 9.2.1]{DR1}); \cite[Prop. 2.1]{I} for generalized vector bundles of  rank $r$; \cite[Thm. 4.6, 4.7]{CK} for $R=2$ and the semistable locus. The irreducibility was also shown for the open subset of quasi-locally free sheaves of fixed complete type by \cite[Thm. C]{DR2}.

\begin{proof}
Fix $(r_\bullet;d_\bullet)\in \cS_{adm}(R,D)$. The locus  $\cM_X(r_{\bullet};d_{\bullet})$ (which is non-empty by Theorem \ref{T:non-empty}) is clearly constructible and it is locally closed by Fact \ref{F:FiltrCan}\ref{F:FiltrCan:5}. We endow the locally closed substack $\cM_X(r_{\bullet};d_{\bullet})$ with its reduced stack structure. 
Consider the following morphisms
$$
\xymatrix{
\cM_X(r_{\bullet};d_{\bullet}) \ar[r]^(0.35){\chi} & \un{\operatorname{Ext}}(\Coh{r_1}{d_1+r_1\delta}, \Coh{r_0-r_1}{d_0-r_1\delta-d_1}) \ar[r]^{\rho} & \Coh{r_0-r_1}{d_0-r_1\delta-d_1}\times \Coh{r_1}{d_1+r_1\delta}\\
\sF \ar[r] & [0\to \frac{\sF^{(1)}}{\sN\sF}\to \sF_{|\Xred}\to \frac{\sF}{\sF^{(1)}}\to 0] & \\
& [0\to \cH'\to \cH \to \cH''\to 0] \ar[r] & (\cH'',\cH'),
}
$$
where $\Coh{n}{e}$ is the stack of rank $r$ and degree $e$ sheaves on $\Xred$, and $\un{\operatorname{Ext}}(\Coh{n''}{e''}, \Coh{n'}{e'})$ is the stack classifying extensions $[0\to \cH'\to \cH \to \cH''\to 0]$ on $\Xred$ with $\cH'\in \Coh{n'}{e'}$ and $\cH''\in \Coh{n''}{e''}$. 

It is well-known that the stack $\Coh{n}{e}$ is smooth and irreducible of dimension equal to $n^2(\ov g-1)$. Hence  $\Coh{r_0-r_1}{d_0-r_1\delta-d_1}\times \Coh{r_1}{d_1+r_1\delta}$ is smooth and irreducible of dimension equal to
\begin{equation}\label{E:dim1}
\dim \left(\Coh{r_0-r_1}{d_0-r_1\delta-d_1}\times \Coh{r_1}{d_1+r_1\delta}\right)=(\ov g-1)[r_1^2+(r_0-r_1)^2].
\end{equation}

By \cite[Cor. 3.2]{GPHS}, the morphism $\rho$ is a vector bundle stack of relative dimension equal to 
\begin{equation}\label{E:dim2}
\dim \rho=(\ov g-1)[r_1(r_0-r_1)]+r_0d_1-r_1d_0+r_0r_1\delta.
\end{equation}

By the proof of \cite[Prop. 5.2]{MS}, the morphism $\chi$ is a vector bundle stack of relative dimension equal to 
\begin{equation}\label{E:dim3}
\dim \chi=(\ov g-1)[r_0r_1]-r_0d_1+r_1d_0.
\end{equation}

By combining \eqref{E:dim1}, \eqref{E:dim2} and \eqref{E:dim3}, we conclude that $\cM_X(r_{\bullet};d_{\bullet})$ is smooth and irreducible of dimension 
$$
\dim \cM_X(r_{\bullet};d_{\bullet})=\dim\left( \Coh{r_0-r_1}{d_0-r_1\delta-d_1}\times \Coh{r_1}{d_1+r_1\delta}\right)+\dim \rho+\dim \chi=(r_0^2+r_1^2)(\ov g -1)+r_0r_1\delta. 
$$
\end{proof}

\begin{cor}\label{C:qlf}
Let $r_0>r_1>0$ and $(d_0,d_1)\in \bZ^2$. Then the general element of $\cM_X(r_0,r_1;d_0,d_1)$ is a quasi-locally free sheaf. 
\end{cor}
\begin{proof}
We have already observed in Remark \ref{R:exis-qlf} that the locus of quasi-locally free sheaves inside $\cM_X(r_0,r_1;d_0,d_1)$ is non-empty. Hence, since $\cM_X(r_0,r_1;d_0,d_1)$ is irreducible by 
Theorem \ref{T:irr-dim}, it is enough to show that the locus of quasi-locally free sheaves inside $\cM_X(r_0,r_1;d_0,d_1)$ is open. Consider the universal sheaf $\mathfrak{F}$ over $\cM_X(r_0,r_1;d_0,d_1)\times X$, which is flat over $\cM_X(r_0,r_1;d_0,d_1)$. Its restriction $\mathfrak{F}_{\Xred}$ to  $\cM_X(r_0,r_1;d_0,d_1)\times X$ is again flat over $\cM_X(r_0,r_1;d_0,d_1)$,
since it has fiberwise constant rank and degree. Therefore, the locus in $\cM_X(r_0,r_1;d_0,d_1)$ over which  $\mathfrak{F}_{\Xred}$ is fiberwise locally free is open. But this is exactly the locus over which $\mathfrak{F}$ is quasi-locally free by Fact \ref{Fact:qll}\ref{Fact:qll:1}, hence we are done. 
\end{proof}

\begin{cor}\label{C:gengvb}
Let $r>0$ and $(d_0,d_1)\in \bZ^2$ such that $d_0\geq d_1+r\delta$. Then the general element $\sF$ of $\cM_X(r,r;d_0,d_1)$ is such that 
\begin{equation}\label{E:gentors}
\begin{sis}
& \cT(\sF_{|\Xred})=\cO_E  \text{ for a reduced divisor } E\subset \Xred \text{ of length } \iota(\sF)=d_0-d_1-r\delta, \\
& n_\bullet(\sF,p)=\{n_1(\sF,p)=1>n_2(\sF,p)=\ldots=n_r(\sF,p)=0\} \quad  \text{ for any } p\in \supp \cT(\sF_{|\Xred}). 
\end{sis}
\end{equation}
\end{cor}
\begin{proof}
Observe that condition \eqref{E:gentors} is equivalent to the fact that  
\begin{equation}\label{E:red-tors}
\cT(\sF_{|\Xred})=\cO_E \quad \text{ for a reduced divisor } E\subset \Xred,
\end{equation} 
because $\cT(\sF_{|\Xred})$ is torsion sheaf on $\Xred$ of length equal to $ \iota(\sF)=d_0-d_1-r\delta$ by Definition \ref{Def:index} and Lemma \ref{L:genVB}, and Remark \ref{Rmk:Def:index} implies that 
$$
\iota(\sF)=\sum_{p\in \supp \cT(\sF_{|\Xred})} \sum_{i=1}^r n_i(\sF,p).
$$
Condition \eqref{E:gentors} is clearly an open condition. The fact that it is also non-empty follows by applying Proposition \ref{P:ex-qufree} in the following way: we take a surjection  $q: \sE\twoheadrightarrow\cO_E $ on $\Xred$  from a rank $r$ and degree $d_0-r\delta$ vector bundle $\sE$ to the structure sheaf of a reduced divisor $E\subset \Xred$ of length $d_0-d_1-r\delta$, and we set  $f: \ker{q}=:\sF\hookrightarrow \sE$ and 
$g:\cO_D\oplus \sF\otimes \sN^{-1}=:\sG\twoheadrightarrow \sF\otimes \sN^{-1}$.
\end{proof}

\subsection{Specializations}\label{sub:speciali}

In this section by determining when the closure of a stratum of $\cM_X(R;D)$ intersects another stratum.

\begin{thm}\label{T:specia}
Let $X$ be a ribbon and fix any $(R, D)\in \bN_{>0} \times \bZ$. 
Assume that $(r_\bullet;d_\bullet), (\wt r_\bullet;\wt d_\bullet)\in \cS_{adm}(R,D)$.  Then 
$$
\ov{\cM_X(r_\bullet; d_\bullet)}\cap \cM_X(\wt r_{\bullet};\wt d_{\bullet})\neq \emptyset \Longleftrightarrow (r_\bullet;d_\bullet)\geq (\wt r_\bullet;\wt d_\bullet).
$$
\end{thm}
\begin{proof}
Assume first that $\ov{\cM_X(r_\bullet; d_\bullet)}\cap \cM_X(\wt r_{\bullet};\wt d_{\bullet})\neq \emptyset$, i.e. that there exists a torsion-free sheaf of complete type $(r_\bullet; d_\bullet)$ that specializes to a torsion-free sheaf of complete type $(\wt r_\bullet; \wt d_\bullet)$. 
Fact \ref{F:FiltrCan}\ref{F:FiltrCan:5} implies that $r_1\geq \wt r_1$. It remains to show that if $r_1=\wt r_1$ (or equivalently if $r_{\bullet}=\wt r_{\bullet}$) then $d_1\geq \wt d_1$ (or equivalently that $d_0\leq \wt d_0$). 

 By assumption, there exists a coherent sheaf $\bF$ on $X\times \Spec R$, where $\Spec R$ is a trait (i.e. $R$ is a DVR) with special point $o$ and  generic point $ \eta$, flat over $\Spec R$, such that  
$\bF_{\eta}\in \cM_X(r_\bullet; d_\bullet)$ and $\bF_o\in \cM_X(\wt r_{\bullet};\wt d_{\bullet})$. Consider the coherent (not necessarily flat) sheaf $\sN\cdot \bF$ on $\Xred\times \Spec R$, which has the property that $(\sN\cdot \bF)_o$ is a vector bundle on $\Xred$ of rank $\wt r_1$ and degree $\wt d_1$, while  $(\sN\cdot \bF)_\eta$ is a vector bundle on $\Xred$ of rank $ r_1$ and degree $d_1$. Pick an ample line bundle $\cO(1)$ on $\Xred$ and set $d:=\deg \cO(1)$. It follows from the proof of \cite[III.9.9]{Har} that the function $s\to \chi((\sN\cdot \bF)_s\otimes \cO(m))$ is lower semicontinuous for $m\gg 0$. Hence we compute
$$
(\wt d_1+\wt r_1dm)+\wt r_1(1-\ov g)= \chi((\sN\cdot \bF)_o\otimes \cO(m))\leq \chi((\sN\cdot \bF)_\eta\otimes \cO(m))=(d_1+r_1dm)+r_1(1-\ov g).
$$
Since $r_1=\wt r_1$ by assumption, we conclude that $\wt d_1\leq d_1$, as required.

Assume now that $(r_\bullet;d_\bullet)\geq (\wt r_\bullet;\wt d_\bullet)$.  We want to show that there exists a coherent sheaf $\bF$ on $X\times \Spec R$, where $\Spec R$ is a trait (i.e. $R$ is a DVR) with special point $o$ and  generic point $ \eta$, flat over $\Spec R$, such that  
$\bF_{\eta}\in \cM_X(r_\bullet; d_\bullet)$ and $\bF_o\in \cM_X(\wt r_{\bullet};\wt d_{\bullet})$. We will distinguish four cases.

\un{Case I:} $\wt r_0=r_0=r_1=\wt r_1$.

Since $(r_{\bullet},d_{\bullet})$ is admissible and $r:=r_0=r_1$, we have that $d_0\geq d_1+r\delta$. Moreover, by the assumption that $(r_\bullet;d_\bullet)\geq (\wt r_\bullet;\wt d_\bullet)$ and the fact that $r_{\bullet}=\wt r_{\bullet}$, we have that 
$$
\wt d_0\geq d_0\geq d_1+r\delta\geq \wt d_1+r\delta.
$$
Hence, we can find integers $\{(d_0^j,d_1^j, \wt d_0^j, \wt d_1^j)\}_{1\leq j \leq r}$ such that 
$$
\begin{sis}
& \wt d_0^j\geq d_0^j\geq d_1^j+\delta \geq \wt d_1^j+\delta \quad \text{ for every } 1\leq j \leq r, \\
& d_0^j+d_1^j=\wt d_0^j+\wt d_1^j \quad \text{ for every } 1\leq j \leq r, \\
&\sum_{j=1}^r d_0^j=d_0 \quad \text{ and } \quad \sum_{j=1}^r \wt{d_0^j}=\wt{d_0}. 
\end{sis}
$$
By a re-iterated application of Lemma \ref{L:def-glb} below, we can find coherent sheaves $\{\bF^j\}_{1\leq j \leq r}$ on $X\times \Spec R$, flat over a trait $\Spec R$, such that  
$\bF^j_{\eta}\in \cM_X(1,1; d_0^j,d_1^j)$ and $\bF^j_o\in \cM_X(1,1;\wt{d_0^j}, \wt{d_1^j})$. The desired family is given by $\bF=\bigoplus_{j=1}^r \bF^j$. 

\un{Case II:} $\wt r_0=r_0>r_1=\wt r_1$.

Set $r:=r_1=\wt r_1$.  We can find integers  $\{(d_0^j,d_1^j, \wt d_0^j, \wt d_1^j)\}_{1\leq j \leq r}$ and $\{e,\wt e\}$ such that 
$$
\begin{sis}
& \wt d_0^j\geq d_0^j\geq d_1^j+\delta \geq \wt d_1^j+\delta \quad \text{ for every } 1\leq j \leq r, \\
& d_0^j+d_1^j=\wt d_0^j+\wt d_1^j \quad \text{ for every } 1\leq j \leq r, \\
&\sum_{j=1}^r d_0^j+e=d_0 \quad \text{ and } \quad \sum_{j=1}^r \wt{d_0^j}+\wt e=\wt{d_0}. 
\end{sis}
$$
By a re-iterated application of Lemma \ref{L:def-glb} below, we can find coherent sheaves $\{\bF^j\}_{1\leq j \leq r}$ on $X\times \Spec R$, flat over a trait $\Spec R$, such that  
$\bF^j_{\eta}\in \cM_X(1,1; d_0^j,d_1^j)$ and $\bF^j_o\in \cM_X(1,1;\wt{d_0^j}, \wt{d_1^j})$. Moreover, take a vector bundle $\cE$ on $\Xred$ of rank $r_0-r_1=\wt r_0-\wt r_1$ and degree $e$. 
The desired family is given by $\bF=\bigoplus_{j=1}^r \bF^j\oplus p_1^*(\cE)$.

\un{Case III:} $\wt r_0>r_0= r_1>\wt r_1$.

Since $(r_{\bullet},d_{\bullet})$ is admissible and $r_0=r_1$, we have that $d_0\geq d_1+r\delta$. Hence, we can find integers $\{(d_0^j,d_1^j)\}_{1\leq j \leq r_1}$ such that 
$$
\begin{sis}
& d_0^j\geq d_1^j+\delta  \quad \text{ for every } 1\leq j \leq r_1, \\
& \sum_{j=r_1-\wt r_1+1}^{r_1} d_1^j=\wt d_1, \\
&\sum_{j=1}^{r_1} d_0^j=d_0. 
\end{sis}
$$
By a re-iterated application of Lemma \ref{L:def-VB2} below, we can find coherent sheaves $\{\bF^j\}_{1\leq j \leq r_1-\wt r_1}$ on $X\times \Spec R$, flat over a trait $\Spec R$, such that  
$\bF^j_{\eta}\in \cM_X(1,1; d_0^j,d_1^j)$ and $\bF^j_o\in \cM_X(2,0;d_0^j+d_1^j,0)$. Moreover, take generalized line bundles $\{\cF^j\}_{j=r_1-\wt r_1+1}^{r_1}$ on $X$ of complete types $(1,1;d_0^j,d_1^j)$. 
The desired family is given by $\bF=\bigoplus_{j=1}^{r_1-\wt r_1} \bF^j\oplus \bigoplus_{j=r_1-\wt r_1+1}^{r_1} p_1^*(\cF^j)$.

\un{Case IV:} $\wt r_0>r_0> r_1>\wt r_1$.

We can find integers  $\{(d_0^j,d_1^j)\}_{1\leq j \leq r_1}$ and $e$ such that 
$$
\begin{sis}
& d_0^j\geq d_1^j+\delta  \quad \text{ for every } 1\leq j \leq r_1, \\
& \sum_{j=r_1-\wt r_1+1}^{r_1} d_1^j=\wt d_1, \\
&\sum_{j=1}^r d_0^j+e=d_0. 
\end{sis}
$$
By a re-iterated application of Lemma \ref{L:def-VB2} below, we can find coherent sheaves $\{\bF^j\}_{1\leq j \leq r_1-\wt r_1}$ on $X\times \Spec R$, flat over a trait $\Spec R$, such that  
$\bF^j_{\eta}\in \cM_X(1,1; d_0^j,d_1^j)$ and $\bF^j_o\in \cM_X(2,0;d_0^j+d_1^j,0)$. Moreover, take generalized line bundles $\{\cF^j\}_{j=r_1-\wt r_1+1}^{r_1}$ on $X$ of complete types $(1,1;d_0^j,d_1^j)$ and 
 a vector bundle $\cE$ on $\Xred$ of rank $r_0-r_1$ and degree $e$. 
 The desired family is given by $\bF=\bigoplus_{j=1}^{r_1-\wt r_1} \bF^j\oplus \bigoplus_{j=r_1-\wt r_1+1}^{r_1} p_1^*(\cF^j)\oplus p_1^*(\cE)$.

\end{proof}

\begin{lemma}\label{L:def-glb}
Let $(d_0,d_1)\in \bZ^2$ such that $d_0\geq d_1+\delta$. Then 
$$\ov{\cM_X((1,1); (d_0,d_1))}\cap \cM_X((1,1); (d_0+1,d_1-1))\neq \emptyset.$$
\end{lemma}
\begin{proof}
This is well-known, see e.g. \cite[Lemma 4.5]{CK}. Let us sketch a proof. 

By our assumptions on $(d_0,d_1)$, we have that 
$$b:=d_0-d_1+\delta+2\geq 2.$$
 Consider an effective divisor $D=2p+E$ on $\Xred$ of length $b$ with $p\not \in \supp E$ and a line bundle $\cL$ on $X$ such that $\deg \cL_{|\Xred}=d_0+1$.
By Example \eqref{Ex:genlb}, the generalized line bundle $\cI_D\otimes \cL$ belongs to  $\cM_X((1,1); (d_0+1,d_1-1))$. There exists a family $\cD\subset X\times \Spec R$ of $0$-dimensional subschemes of $X$ of length $b$ 
over a trait $\Spec R$ such that the special fiber $\cD_o$ is equal to $D=2p+E$ and the general fiber $\cD_{\eta}$ is equal to $E+C$ where $C$ is a  $0$-dimensional subscheme of $X$ of length two supported at $p$ and not entirely contained in $\Xred$. 
It follows that $\cI_{\cD_{\eta}}=\cI_E\otimes \cI_C$, where $\cI_C$ is a line bundle on $\Xred$ such that $(\cI_C)_{|\Xred}=\cO_{\Xred}(-p)$.
 Therefore, the family of torsion-free sheaves $\cI_{\cD}\otimes p_1^*(\cL)$ on $X\times \Spec R$ is such that $(\cI_{\cD}\otimes p_1^*(\cL))_o=\cI_D\otimes \cL$ while the generic fiber $(\cI_{\cD}\otimes p_1^*(\cL))_{\eta}=\cI_E\otimes \cI_C\otimes \cL$ belongs 
 to $\cM_X((1,1); (d_0,d_1))$ by Example \eqref{Ex:genlb}. This concludes the proof.
 \end{proof}

\begin{lemma}\label{L:def-VB2}
Let $(d_0,d_1)\in \bZ^2$ such that $d_0\geq d_1+\delta$. Then 
$$\ov{\cM_X((1,1); (d_0,d_1))}\cap \cM_X((2,0); (d_0+d_1,0))\neq \emptyset.$$
\end{lemma}
\begin{proof}
By our assumptions on $(d_0,d_1)$, we can find two line bundles $L_0$ and $L_1$ on $\Xred$ of degree, respectively, $d_0-\delta$ and $d_1+\delta$ such that there exists a injective map $\phi:L_1\otimes \sN\hookrightarrow L_2$. 
Consider the rank two vector bundle $\cE=L_1\oplus L_2$ on $\Xred$, which has degree $d_0+d_1$ and sits in the split exact sequence
$$
0\to L_0\to \cE\to L_1\to 0.
$$
Since there exists an injective map $\phi:L_1\otimes \sN\hookrightarrow L_2$ and the condition of being injective is open for families of morphisms, the generic element $[\cF]\in \Ext{1}(L_1,L_0)$ will have an associated injective morphism $\pi([\cF]): L_1\otimes \sN\hookrightarrow L_2$  (with the notation of Fact \ref{F:FiltrCan}\ref{F:FiltrCan:3}). Therefore, the same Fact \ref{F:FiltrCan}\ref{F:FiltrCan:3} implies that, for the general element  $[\cF]\in \Ext{1}(L_1,L_0)$, the extension
$$
0\to L_0 \to \cF \to L_1\to 0
$$
is the second canonical filtration. Using Remark \ref{R:inv-F1}, we get that the generic element $[\cF]\in \Ext{1}(L_1,L_0)$ has complete type equal to 
$$
(r_0(\cF),r_1(\cF);d_0(\cF),d_1(\cF))=(1,1;d_0, d_1),
$$
and this concludes the proof. 
\end{proof}

\section{Irreducible components}\label{Sec:Irred}

In this section, we determine the irreducible components of $\cM_X(R,D)$. There are two different cases, according to whether $\delta\leq 2\ov g-2$ or $\delta> 2\ov g-2$.
The first case is easier. 

\begin{prop}\label{P:irrcomp1}
 If $\delta \leq 2(\ov g-1)$ then the irreducible components of $\cM_X(R,D)$ are
$$
\{\ov{\cM_X(r_{\bullet};d_{\bullet})}\}_{(r_{\bullet};d_{\bullet})\in \cS_{adm}(R,D)}.
$$
Moreover:
\begin{enumerate}
\item If $\delta < 2(\ov g-1)$ then $\cM_X(R,D)$ has dimension $R^2(\ov g-1)$ and the unique irreducible component of maximal dimension is $\cM_X(R,0;D,0)$, i.e. the component that parametrizes vector bundles on $\Xred$ of rank $R$ and degree $D$. 
\item If $\delta= 2\ov g-2$ then $\cM_X(R,D)$ is pure of dimension equal to $R^2(\ov g-1)$. 
\end{enumerate}
\end{prop}
This was proved for  $R=2$ (for the semistable locus) in \cite[Thm. B]{CK}. 
\begin{proof}
The first assertion  follows from Theorem \ref{T:specia} together with the fact that, for any $(r_\bullet;d_\bullet), (\wt r_\bullet;\wt d_\bullet)\in \cS_{adm}(R,D)$ and 
 $\delta \leq 2(\ov g-1)$, we have that 
$$(r_\bullet;d_\bullet)\geq (\wt r_\bullet;\wt d_\bullet)\Rightarrow \wt r_0\geq r_0\geq r_1 \geq \wt r_1\Rightarrow \dim \cM_X(r_\bullet; d_\bullet)\leq \dim \cM_X(\wt r_{\bullet};\wt d_{\bullet}),$$
as it follows from Theorem \ref{T:irr-dim}. 

The second assertion follows from the first one together with Theorem \ref{T:irr-dim}.
\end{proof}

In order to determine the irreducible components of $\cM_X(R,D)$ in the case where $\delta>2\ov g-2$,  the key ingredient that we will use is the following deformation criterion.

\begin{prop}\label{P:def-seq}
Let $\sF$ be a torsion-free sheaf on $X$. Assume that there exists a vector bundle $K$ on $\Xred$ such that
\begin{enumerate}[(i)]
\item \label{Ipo1} $\sN\sF\subseteq K\subseteq \sF^{(1)}$; 
\item \label{Ipo2} there exists an injective morphism 
$$\phi: \frac{\sF}{K}\otimes \sN \hookrightarrow K.$$
\end{enumerate}
Then $\sF$ can be deformed to a torsion-free sheaf $\cE$ on $X$ such that 
\begin{equation}\label{E:IIcan-E}
\sE^{(1)}=K \quad \text{ and } \quad \frac{\sE}{\sE^{(1)}}= \frac{\sF}{K}.
\end{equation}
\end{prop}
\begin{proof}
By Fact \ref{F:FiltrCan}\ref{F:FiltrCan:3} and the assumption \eqref{Ipo2}, the generic extension 
$$
\sigma=[0\to K\to \sE \to \frac{\sF}{K}\ \to 0]\in \Ext{1}_{\cO_{X}}\left(\frac{\sF}{K},K\right)
$$
is such that the associated morphism $\displaystyle \pi(\sigma):\frac{\sF}{K}\otimes \sN \rightarrow K$ is injective. 
Hence, Fact \ref{F:FiltrCan}\ref{F:FiltrCan:3} implies that the generic such extension satisfies properties \eqref{E:IIcan-E}. 
We conclude by taking a generic deformation of $[0\to K\to \sF \to \frac{\sF}{K}\to 0]$ inside  $\Ext{1}_{\cO_{X}}\left(\frac{\sF}{K},K\right)$.
\end{proof}

\begin{thm}\label{T:def-rig}
Let $X$ be a ribbon such that $\delta>2\ov g-2$.  For any $ (\wt r_\bullet;\wt d_\bullet)\in \cS_{adm}(R,D)$, there exists $(r_\bullet;d_\bullet)\in \cS_{rig}(R,D)$ such that 
$$  \cM_X(\wt r_{\bullet};\wt d_{\bullet})\subseteq \ov{\cM_X(r_\bullet; d_\bullet)}.$$
\end{thm}
The above result was proved for $R=2$ in \cite[Thm. 1]{S2}.
\begin{proof}
We can assume that $ (\wt r_\bullet;\wt d_\bullet)$ is not of rigid type, i.e. $r_0\geq r_1+2$, for otherwise there is nothing to prove. 
Take a general element $\sF\in \cM_X(\wt r_\bullet;\wt d_\bullet)$, which is in particular quasi-locally free by Corollary \ref{C:qlf}.
We will use Proposition \ref{P:def-seq} to deform $\sF$ to a torsion-free sheaf $\sE$ on $X$ of complete rank $(r_0-1,r_1+1)$, which is enough to conclude the proof (by iterating the argument). 

In order to show this, take a subvector bundle $\ov K$ of rank $r_0-r_1-1$ of the rank $r_0-r_1$ and degree $d_0-d_1-r_1\delta$ vector bundle $\Gamma_{\sF}:=\frac{\sF^{(1)}}{\sN\sF}$, which has 
 maximal degree among the rank $r_0-r_1-1$  subvector bundles of $\Gamma_{\sF}$. By definition, we have that 
 \begin{equation}\label{E:SegreG}
 (r_0-r_1-1)[d_0-d_1-r_1\delta]-(r_0-r_1)\deg \ov K=s_{r_0-r_1-1}(\Gamma_{\sF}),
 \end{equation}
where $s_{r_0-r_1-1}(\Gamma_{\sF})$ is, by definition, the $(r_0-r_1-1)$-Segre invariant of $\Gamma_{\sF}$  (see \cite{BPL, RT}). Since the Segre invariant is lower semicontinuous (see \cite{BPL, RT}), we have that 
\begin{equation}\label{E:Seg-lower}
s_{r_0-r_1-1}(\Gamma_{\sF})\leq s_{r_0-r_1-1}(r_0-r_1,d_0-d_1-r_1\delta), 
\end{equation}
where $s_{r_0-r_1-1}(r_0-r_1,d_0-d_1-r_1\delta)$ is the $(r_0-r_1-1)$-Segre invariant of the general rank $r_0-r_1$ and degree $d_0-d_1-r_1\delta$ vector bundle on $\Xred$.
Moreover, the general such Segree invariant is the unique integer such that  (see e.g. \cite[Thm. 0.1]{RT}\footnote{This is proved in loc. cit. under the assumption that $\ov g\geq 2$. However, the result is easily seen to be true also for $\ov g=0, 1$ by using the Segre-Grothendieck splitting theorem for vector bundles on $\bP^1$ and the Atiyah's classification of vector bundles on elliptic curves.}) 
 \begin{equation}\label{E:Seg-gen}
 \begin{sis}
& s_{r_0-r_1-1}(r_0-r_1,d_0-d_1-r_1\delta)\in [(r_0-r_1-1)(\ov g-1),(r_0-r_1-1)\ov g], \\
& s_{r_0-r_1-1}(r_0-r_1,d_0-d_1-r_1\delta)\equiv -(d_0-d_1-r_1\delta)\mod r_0-r_1.
\end{sis}
\end{equation}

Let $\sN \sF\subseteq K\subseteq \sF^{(1)}$ be the inverse image of $\ov K$. We have two exact sequences of vector bundles on $\Xred$
\begin{equation}\label{E:2seq-N}
\begin{aligned}
& 0 \to \sN\sF \to K \to \ov K\to 0,\\
& 0\to \frac{\Gamma_{\sF}}{\ov K}\otimes \sN \to \frac{\sF}{K}\otimes \sN\to \frac{\sF}{\sF^{(1)}}\otimes \sN \to 0.
\end{aligned}
\end{equation}
Since we have an isomorphism $\frac{\sF}{\sF^{(1)}}\otimes \sN= \sN\sF$ (see Fact \ref{F:FiltrCan}\ref{F:FiltrCan:1}), we can compute the difference between the slopes of $K$ and of 
$\frac{\sF}{K}\otimes \sN$ as it follows (using \eqref{E:SegreG})
\begin{equation}\label{E:difslopes}
\mu(K)-\mu\left(\frac{\sF}{K}\otimes \sN\right)=\mu(\ov K)-\mu\left(\frac{\Gamma_{\sF}}{\ov K}\otimes \sN\right)=\frac{\deg \ov K}{r_0-r_1-1}-[d_0-d_1-r_1\delta-\deg \ov K-\delta]=
\end{equation}
$$
=\frac{(r_0-r_1)\deg \ov K}{r_0-r_1-1}-[d_0-d_1-r_1\delta]+\delta=-\frac{s_{r_0-r_1-1}(\Gamma_{\sF})}{r_0-r_1-1}+\delta.
$$
Now, using \eqref{E:Seg-lower} and \eqref{E:Seg-gen} and our assumption that $\delta > 2\ov g-2$ , we distinguish two cases.

\un{Case I:} Either $s_{r_0-r_1-1}(\Gamma_{\sF})<(r_0-r_1-1)\ov g$ or $\delta\geq 2\ov g$. 

In this case, \eqref{E:difslopes} implies that 
$$
\mu(K)-\mu\left(\frac{\sF}{K}\otimes \sN\right)>\ov g-1.
$$
By the twisted higher Brill-Noether rank theory (see \cite[Thm. 0.3]{RT}), this implies that there exists an injective morphism 
$$
\phi: \frac{\sF}{K}\otimes \sN\hookrightarrow K.
$$
Hence Proposition \ref{P:def-seq} implies that the sheaf $\sF$ can be deformed to a torsion free sheaf $\sE$ such that 
$$\sE^{(1)}=K \quad \text{ and } \quad \frac{\sE}{\sE^{(1)}}= \frac{\sF}{K}.$$
In particular, $\sE$ will have complete rank $(r_0-1,r_1+1)$, as required. 

\un{Case II:} $s_{r_0-r_1-1}(\Gamma_{\sF})=(r_0-r_1-1)\ov g$
and $\delta=2\ov g-1$.

In this case, \eqref{E:difslopes} implies that 
$$
\mu(K)-\mu\left(\frac{\sF}{K}\otimes \sN\right)=\ov g-1.
$$
Moreover, it follows from \cite[Thm. 4.4]{BPL} or \cite[Thm. 0.2]{RT} that 
$$
\dim \left\{\ov K\subset \Gamma_{\sF}: \rk{\ov K}=r_0-r_1-1 \text{ and } \deg(\ov K)=\frac{(r_0-r_1-1)(d_0-d_1-r_1\delta-\ov g)}{r_0-r_1} \right\}\geq r_0-r_1-1.
$$
Again by the twisted higher Brill-Noether rank theory (see \cite[Thm. 0.3]{RT}), this implies that  we can choose such a $K$ with the property that there exists an injective morphism 
$$
\phi: \frac{\sF}{K}\otimes \sN\hookrightarrow K.
$$
Hence, we conclude as in the previous Case. 
\end{proof}

\begin{cor}\label{C:irrcomp2}
 If $\delta > 2(\ov g-1)$ then the irreducible components of $\cM_X(R,D)$ are
$$
\{\ov{\cM_X(r_{\bullet};d_{\bullet})}\}_{(r_{\bullet};d_{\bullet})\in \cS_{rig}(R,D)}.
$$
In particular $\cM_X(R,D)$ has pure dimension equal to 
$$\dim \cM_X(R,D)=
\begin{cases}
\frac{R^2}{4}(2\ov g-2+\delta)= \frac{R^2}{4}(g-1) & \text{ if } R \text{ is even,} \\
\frac{R^2+1}{4}(2\ov g-2)+\frac{R^2-1}{4}\delta  & \text{ if } R \text{ is odd.} 
\end{cases}
$$
\end{cor}
The above result was proved for $R=2$ (for the semistable locus) in \cite[Cor. 1]{S2} (completing \cite[Thm. B]{CK}).
\begin{proof}
This follows by combining Theorem \ref{T:def-rig} and Theorem \ref{T:irr-dim}. 
\end{proof}

\section{The (semi)stable locus}\label{Sec:ss-Locus}

The aim of this section is to study the open substack of the stack $\cM_X$ parametrizing (slope-) stable or semistable torsion-free sheaves, whose definition we now recall.

\begin{defi}\label{Def:SemiStab}
Let $\sF$ be a torsion-free sheaf on a ribbon $X$ with $R(\sF)>0$ (or, equivalently, non zero) and define its   \emph{slope} by
$$\mu(\sF)= \frac{\Deg{\sF}}{\Rk{\sF}}.$$ 
 \begin{enumerate}[(1)]
 \item We say that:
\begin{itemize}
\item  $\sF$ is \emph{semistable} if $\mu(\sG)\le\mu(\sF)$ for any proper subsheaf $0\neq \sG\subsetneq \sF$; 
\item $\sF$ is \emph{stable} if $\mu(\sG)<\mu(\sF)$ for any proper subsheaf $0\neq \sG\subsetneq \sF$; 
\item $\sF$ is \emph{polystable} if it is equal to the direct sum of stable sheaves of the same slopes. 
\end{itemize}
\item 
Let $\sF$ be a semistable sheaf, then a \emph{Jordan-Holder filtration} of $\sF$ is a filtration whose associated graded object $\Gr{\text{JH}}{\sF}$ is polystable of the same slope of $\sF$. It is well known that any semistable sheaf admits a (non-necessarily unique) Jordan-Holder filtration and that $\Gr{\text{JH}}{\sF}$ is independent of the choice of the filtration (see e.g. \cite[Proposition 1.5.2]{HL}). Clearly, if $\sF$ is stable, $\Gr{\text{JH}}{\sF}=\sF$.
Two semistable sheaves $\sF$ and $\sG$ are said to be \emph{S-equivalent} if $\Gr{\text{JH}}{\sF}$ and $\Gr{\text{JH}}{\sG}$ are isomorphic.  
\end{enumerate}
\end{defi}

\begin{rmk}\label{R:SStab}
\noindent
\begin{enumerate}
\item\label{R:SStab:1} As usual, it is possible to check semistability considering only saturated subsheaves of $\sF$ (i.e. subsheaves $\sG$ such that the quotient $\sF/\sG$ is torsion-free) and it is equivalent to the reverse inequalities for quotients of $\sF$ (these are well-known basic properties of semistability, see e.g. \cite[Proposition 1.2.6]{HL}).

\item\label{R:SStab:2} It follows from Facts \ref{F:Dual}\ref{F:Dual:4} and \ref{F:Dual}\ref{F:Dual:4bis} and from the previous point that a torsion-free sheaf $\sF$ is (semi)stable if and only if its dual $\sF^{\vee}$ is (semi)stable.

\item\label{R:SStab:3} By Fact \ref{F:GenRkDeg}\ref{F:GenRkDeg:1}, slope-semistability is equivalent, on a ribbon $X$, to Gieseker semistability with respect to any ample line bundle on $X$ (which is defined in terms of the leading coefficient of the reduced Hilbert polynomial, for its precise definition see \cite{HL}).

\end{enumerate}

\end{rmk}

It is well-known that stability and semistability are open conditions, so they determine open substacks of the stack of torsion-free sheaves on $X$. 

\begin{defi}\label{D:sslocus} 
Denote by 
$$\cM_X(R,D)^s\subseteq \cM_X(R,D)^{ss}\subseteq \cM_X(R,D)$$ 
 the open substack parametrizing stable (resp. semistable) sheaves on $X$ of generalized rank $R$ and generalized degree $D$.
\end{defi}

Using that slope-(semi)stability is equivalent on a ribbon to Gieseker (semi)-stability, the stacks of (semi)stable torsion-free sheaves admit suitable moduli spaces.

\begin{fact}(see \cite{Si}, \cite{HL})\label{F:ModSpa}
There exists a Cartesian diagram 
$$
\xymatrix{
\cM_X(R,D)^s\ar@{^{(}->}[r] \ar^{\Phi^s}[d] &  \cM_X(R,D)^{ss}\ar^{\Phi^{ss}}[d] \\
M_X(R,D)^s\ar@{^{(}->}[r] & M_X(R,D)^{ss}
} 
$$
with the following properties
\begin{enumerate}
\item the horizontal arrows are open embeddings.
\item the stacks $\cM_X(R,D)^s$ and $\cM_X(R,D)^{ss}$ are of finite type over $k$.
\item the morphism $\Phi^{ss}$ is an adequate moduli space (and a good moduli space if $\rm{char}(k)=0$) and given two points $\sF,\sG\in  \cM_X(R,D)^{ss}(k)$ we have that $\Phi^{ss}(\sF)=\Phi^{ss}(\sG)$ if and only if $\sF$ and $\sG$ are S-equivalent. In particular, the $k$-points of $M_X(R,D)^{ss}$ are in bijection with polystable sheaves of generalized rank $R$ and generalized degree $D$.
\item the morphism $\Phi^s$ is a trivial $\Gm$-gerbe, i.e.  $\cM_X(R,D)^{s}\cong M_X(R,D)^s\times B\Gm$. 
\item $M_X(R,D)^{ss}$ is a projective $k$-scheme (and hence $M_X(R,D)^{s}$ is a quasi-projective $k$-scheme).
\end{enumerate}
\end{fact}

We now determine which strata of $\cM_X(R,D)$ intersect the semistable or stable locus.

\begin{thm}\label{T:ss-strata}
Let $(r_\bullet;d_\bullet)\in \cS_{adm}(R,D)$.
\begin{enumerate}[(1)]
\item \label{T:ss-strata1} If the stratum $\cM_X(r_\bullet; d_\bullet)$ intersects $\cM_X(R,D)^{ss}$ then one the following occurs:
\begin{enumerate}[(a)]
\item $r_1=0$;
\item $r_0>r_1>0$ and $\displaystyle \frac{d_0-(r_0+r_1)\delta}{r_0}\le\frac{d_1}{r_1}\le\frac{d_0}{r_0}$;
\item $r_0=r_1$ and $d_0\leq d_1+2r_0\delta$.
\end{enumerate}
\vspace{0.1cm} 
If $\ov g\geq 1$ then the above numerical conditions are also sufficient for the non-emptiness of $\cM_X(r_\bullet; d_\bullet)\cap \cM_X(R,D)^{ss}$.
\item \label{T:ss-strata2} If the stratum $\cM_X(r_\bullet; d_\bullet)$ intersects $\cM_X(R,D)^{s}$ then one the following occurs:
\begin{enumerate}[(a)]
\item $r_1=0$;
\item $r_0>r_1>0$ and $\displaystyle \frac{d_0-(r_0+r_1)\delta}{r_0}<\frac{d_1}{r_1}<\frac{d_0}{r_0}$;
\item $r_0=r_1$ and $d_0< d_1+2r_0\delta$.
\end{enumerate}
\vspace{0.1cm} 
If $\ov g\geq 2$ then the above numerical conditions are also sufficient for the non-emptiness of $\cM_X(r_\bullet; d_\bullet)\cap \cM_X(R,D)^{s}$.
\end{enumerate}
\end{thm}
The above Theorem generalizes and improves some known special cases: \cite[\S 5.2]{DR3} for vector bundles; \cite[\S 5.3]{DR3} for quasi-locally free sheaves of complete rank $(r_0,r_1)=(a+1,a)$ with $a\geq 1$ (extending the case $(r_0,r_1)=(2,1)$ treated in \cite[\S 9.1]{DR1}); \cite[Thm. A]{CK} for $R=2$.

We will separated the proof in three parts.

\begin{proof}[Proof of the first statements in \eqref{T:ss-strata1} and \eqref{T:ss-strata2}]
The first statement in \eqref{T:ss-strata1} follows from the fact that:
\begin{itemize}
\item if $\sF$ is a semistable sheaf of complete type $(r_\bullet; d_\bullet)$ with $r_0>r_1>0$ then we must have (using Definition \ref{D:CanFiltr} and Remark \ref{R:inv-F1}) 
 \begin{equation}\label{E:ineq1}
 \begin{sis}
 & \frac{d_1}{r_1}=\mu(\sN\sF) \leq \mu(\sF)=\frac{d_0+d_1}{r_0+r_1} \left(\Longleftrightarrow \frac{d_1}{r_1}\geq \frac{d_0}{r_0}\right), \\
 & \frac{d_0-r_1\delta}{r_0}=\mu(\sF^{(1)})\leq \mu(\sF)=\frac{d_0+d_1}{r_0+r_1} \left(\Longleftrightarrow  \frac{d_0-(r_0+r_1)\delta}{r_0}\le\frac{d_1}{r_1}\right). 
 \end{sis}
 \end{equation}
\item if $\sF$ is a semistable sheaf of complete type $(r_\bullet; d_\bullet)$ with $r_0=r_1>0$ then we must have 
 \begin{equation}\label{E:ineq2}
 \frac{d_0-r_1\delta}{r_0}=\mu(\sF^{(1)})\leq \mu(\sF)=\frac{d_0+d_1}{r_0+r_1} \left(\Longleftrightarrow d_0\leq d_1+2r_0\delta\right). 
 \end{equation}
 [Note that in this case the inequality $\mu(\sN\sF) \leq \mu(\sF)$ is superfluous since $\sF^{(1)}$ is the saturation of $\sN\sF$ inside $\sF$.]
\end{itemize}

The first statement in \eqref{T:ss-strata2} is proved similarly using strict inequalities in \eqref{E:ineq1} and \eqref{E:ineq2}. 

\end{proof}

\begin{proof}[Proof of the second part of\eqref{T:ss-strata2}] 
We will distinguishing three cases.

$\fbox{Case 1: $r_1=0$}$ This follows from Lemma \ref{L:VBsm}\eqref{L:VBsm2:1}. 

$\fbox{Case 2: $r_0=r_1:=r$}$ 

Assuming that  the strict inequalities in \eqref{E:ineq2} are satisfied, we want to construct a stable generalized vector bundle of complete type $(r,r;d_0,d_1)$.

Lemma \ref{L:VBsm}\eqref{L:VBsm2:3} implies that there exists a stable vector bundle $\cE$ on $\Xred$ of rank $r$ and degree $d_0-r\delta$,  an effective divisor  $Z$ on $\Xred$ of degree $b:=d_0-d_1-r\delta\geq 0$ and a surjection $\phi:E\to \cO_Z$ such that $\cF:=\ker{\phi}$ is a stable vector bundle on $\Xred$ of rank $r_1$ and degree $d_1$.
Hence, we get an exact sequence on $\Xred$
\begin{equation}\label{E:ex-stab}
0\to \cF\to \cE\to \cG:=\cO_Z\oplus (\cF\otimes \sN^{-1})\to \cF\otimes \sN^{-1}\to 0,
\end{equation}
such that $\cE$ and $\cF$ are stable vector bundles. 

Proposition \ref{P:ex-qufree} implies that there exists a coherent sheaf $\sF$ on $\Xred$ such that  its associated canonical exact sequence \eqref{Eq:ExSeq1} is isomorphic to \eqref{E:ex-stab}.
By construction and Fact \ref{F:Dual}\ref{F:Dual:1},  $\sF$ is a generalized vector bundle of complete type $(r,r;d_0,d_1)$. By the numerical assumption \eqref{E:ineq2} and the stability of $\cF$ and $\cE$, we get 
\begin{equation}\label{E:propVB1}
\begin{sis}
& (\sN\sF)^{\rm sat}=(\sF)^{(1)}=\cE \text{ is stable, }\\
& (\sF_{|\Xred})^{\vee\vee}=\sF/\sF^{(1)}=\sN\sF\otimes \sN=\cF\otimes \sN \text{ is stable, }\\
& \mu(\sF^{(1)})<\mu(\sF).
\end{sis}
\end{equation} 
By Fact \ref{F:Dual}\ref{F:Dual:5}, the dual sheaf $\sF^{\vee}$ is again a generalized vector bundle of generalized rank $2r$ and it satisfies the following properties
\begin{equation}\label{E:propVB2}
\begin{sis}
& (\sN\sF^\vee)^{\rm sat}=(\sF^\vee)^{(1)}=(\sF/\sF^{(1)})^\vee=(\cF\otimes \sN^{-1})^\vee=\cF^*\otimes \sN^2 \text{ is stable, }\\
& (\sF^\vee_{|\Xred})^{\vee\vee}=\sF^\vee/(\sF^\vee)^{(1)}=(\sF^{(1)})^{\vee}=\cE^\vee=\cE^*\otimes \sN \text{ is stable, }\\
& \mu((\sF^\vee)^{(1)})=\mu((\sF/\sF^{(1)})^\vee)<\mu(\sF^\vee),
\end{sis}
\end{equation} 
where the last inequality is the dual of the inequality  \eqref{E:ineq2}. We can now apply \cite[Thm. 5.1.2]{DR3}, whose hypothesis are satisfied by \eqref{E:propVB1} and \eqref{E:propVB2}, 
in order to conclude that $\sF$ is stable. 

$\fbox{Case 3: $r_0>r_1>0$}$ 

Assuming that  the strict inequalities in \eqref{E:ineq1} are satisfied, we want to construct a stable torsion-free sheaf of complete type $(r_0,r_1;d_0,d_1)$.

We will distinguish three cases.


\un{Case A:} $d_0/r_0-\delta<d_1/r_1<(d_0-r_1\delta)/r_0$.

Because of these numerical assumptions, we can apply  Lemma \ref{L:VBsm}\eqref{L:VBsm2:2} in order to find  
a stable vector bundle $\cF$ (resp. $\cK$) on $\Xred$ of rank $r_1$ (resp. $r_0-r_1$) and degree $d_1$ (resp. $d_0-d_1-r_1\delta$)
together with two extensions 
$$
\begin{sis}
& 0\to \cF\to \cE \to \cK\to 0, \\
& 0 \to \cK\to \cG\to \cF\otimes \sN^{-1}\to 0,
\end{sis}
$$
such that $\cE$ and $\cG$ are stable vector bundles of rank $r_0$ and degree, respectively, $d_0-r_1\delta$ and $d_0$. 
Hence we get the following exact sequence of stable vector bundles on $\Xred$:
\begin{equation}\label{E:seq-stab}
0\to \cF\to \cE \to \cG\to \cF\otimes \sN^{-1}\to 0. 
\end{equation}
Proposition \ref{P:ex-qufree} implies that there exists a coherent sheaf $\sF$ on $\Xred$ such that  its associated canonical exact sequence \eqref{Eq:ExSeq1} is isomorphic to $\eqref{E:seq-stab}$.
By construction and  Fact \ref{Fact:qll}\ref{Fact:qll:1},  $\sF$ is quasi-locally free and of complete type $(r_0,r_1;d_0,d_1)$.

Observe that, by our assumptions of the stability of the vector bundles in \eqref{E:seq-stab} together with the strict inequalities of \eqref{E:ineq1} and Fact \ref{Fact:qll}\ref{Fact:qll:2}, we have that 
\begin{equation}\label{E:many-in}
\begin{sis}
& \sN\sF=\cF \: \text{ is stable,} \\
& \sF_{|\Xred}=\cG\:  \text{ is stable,}\\
& \sN\sF^{\vee}=(\sN\sF)^{\vee}\otimes \sN=\cF^{\vee}\otimes \sN\: \text{ is stable,} \\
& (\sF^{\vee})_{|\Xred}=(\sF^{(1)})^\vee=\cE^{\vee} \: \text{ is stable, }\\
& \mu(\sN\sF) < \mu(\sF), \\
 & \mu(\sF^{(1)})< \mu(\sF) \Rightarrow  \mu(\sN\sF^\vee) < \mu(\sF^\vee).
\end{sis}
\end{equation}
 We can now apply \cite[Thm. 5.1.2]{DR3}, whose hypothesis are satisfied by \eqref{E:many-in}, in order to conclude that $\sF$ is stable.

\un{Case B:} $[d_0-(r_0+r_1)\delta]/r_0< d_1/r_1\le d_0/r_0-\delta$.

Because of these numerical assumptions, we can apply  Lemma \ref{L:VBsm}\eqref{L:VBsm2:2} in order to find  
a stable vector bundle $\cF$ (resp. $\cK$) on $\Xred$ of rank $r_1$ (resp. $r_0-r_1$) and degree $d_1$ (resp. $d_0-d_1-r_1\delta$)
together with an extension on $\Xred$ 
$$
0\to \cF\to \cE \to \cK\to 0
$$
such that $\cE$ is a stable vector bundles of rank $r_0$ and degree $d_0-r_1\delta$. 
Hence we get the following exact sequence of  vector bundles on $\Xred$:
\begin{equation}\label{E:seq-stab2}
0\to \cF\to \cE \to \cG:=\cK\oplus  (\cF\otimes \sN^{-1}) \to \cF\otimes \sN^{-1}\to 0,
\end{equation}
where $\cF$, $\cE$ and $\cK$ are stable vector bundles on $\Xred$ of ranks, respectively, $r_1$, $r_0$ and $r_0-r_1$ and of degree, respectively, $d_1$, $d_0-r_1\delta$ and $d_0-d_1-r_1\delta$. 

Proposition \ref{P:ex-qufree} implies that there exists a coherent sheaf $\sF$ on $\Xred$ such that  its associated canonical exact sequence \eqref{Eq:ExSeq1} is isomorphic to \eqref{E:seq-stab2}.
By construction and Fact \ref{Fact:qll}\ref{Fact:qll:1},  $\sF$ is quasi-locally free and of complete type $(r_0,r_1;d_0,d_1)$.

 Let us prove that $\sF$ is stable. To this aim, consider  a saturated subsheaf $0\neq  \sG\subsetneq \sF$ and let us show that $\mu(\sG)< \mu(\sF)$, or equivalently that 
 $\mu(\sF)< \mu(\sF/\sG)$.
 
If $\sG\subseteq\sF^{(1)}=\cE$ then we conclude since
$$\mu(\sG)\le \mu(\cE)=\mu(\sF^{(1)}) < \mu(\sF)$$
by the stability of $\cE$ and the numerical assumption \eqref{E:ineq1}.
Analogously, if $\cK\oplus (\cF\otimes \sN^{-1})=\sF|_{\Xred}\surj \sF/\sG$ then we conclude since 
$$
\mu(\sF)> \mu(\sF|_{\Xred}) \geq \mu(\cF\otimes \sN^{-1}) \geq \mu(\sF/\sG)
$$
where the first inequality follows from the numerical assumption \eqref{E:ineq1}, and the last two inequalities follow from the fact that $\cK$ and $\cF\otimes \sN^{-1}$ are stable with slopes 
$\mu(\cK)\geq \mu(\cF\otimes \sN^{-1})$. 

So, we may assume (see Fact \ref{F:FiltrCan}\ref{F:FiltrCan:2}) that nor $\sG$ neither $\sF/\sG$ are defined over $\Xred$. We now distinguish two cases:

$\bullet$ Assume that $\displaystyle \frac{r_0+r_1}{r_0(\sG)+r_1(\sG)}\ge \frac{r_0}{r_0(\sG)}$ (which is equivalent to $r_1r_0(\sG)\geq r_0r_1(\sG)$).

 Fact \ref{F:FiltrCan}\ref{F:FiltrCan:4:2} implies that we have injections $0\subsetneq\sG^{(1)}\subseteq\sF^{(1)}=\cE$ and $0\subsetneq\sG/\sG^{(1)}\subseteq\sF/\sF^{(1)}=\cF\otimes \sN^{-1}$.
 By the stability of $\cE$ and of $\cF$, we get that  $\mu(\sG^{(1)})\le\mu(\sF^{(1)})$ and $\mu(\sG/\sG^{(1)})\le\mu(\sF/\sF^{(1)})$. 
 Moreover we have that $\mu(\sF^{(1)})<\mu (\sF)<\mu(\sF/\sF^{(1)})$ by the numerical assumption \eqref{E:ineq1}. Therefore we conclude that 
 $\mu(\sG)< \mu(\sF)$ by applying Lemma \ref{Lem:VarLemDr} with 
 $$
 \begin{sis}
 & (R_2,R_3,R_5,R_6)=(\rk{\sG/\sG^{(1)}}=r_1(\sG), \rk{\sG^{(1)}}=r_0(\sG), \rk{\sF/\sF^{(1)}}=r_1, \rk{\sF^{(1)}}=r_0), \\
 & (D_2,D_3,D_5,D_6)=(\deg(\sG/\sG^{(1)}), \deg(\sG^{(1)}), \deg(\sF/\sF^{(1)}), \deg(\sF^{(1)})). 
 \end{sis}
 $$

$\bullet$ Assume that $\displaystyle \frac{r_0+r_1}{r_0(\sG)+r_1(\sG)}< \frac{r_0}{r_0(\sG)}$ (which is equivalent to $r_1r_0(\sG)< r_0r_1(\sG)$).

Consider the following commutative diagram:
\[
\begin{tikzcd}
\sN\sG \ar[d, hook] \ar[r, hook] & \sG^{(1)}\ar[d, hook]\ar[r, twoheadrightarrow] &\sG^{(1)}/\sN\sG \ar[d, "\varphi"]\\
\sN\sF=\cF \ar[r, hook] & \sF^{(1)}=\cE \ar[r, twoheadrightarrow] &\sF^{(1)}/\sN\sF=\cK
\end{tikzcd}
\]
It induces the following exact sequence
$$
0\to \cJ:=\sG^{(1)}\cap \sN\sF\to \sG^{(1)} \to \cH:=\im{\varphi}\to 0. 
$$
Therefore we can compute the slope of $\sG$ as
\begin{equation}\label{E:slG}
\mu(\sG)=\frac{\rk{\cJ}\mu(\cJ)+[\rk{\sG^{(1)}}-\rk{\cJ}]\mu(\cH)+[R(\sG)-\rk{\sG^{(1)}}]\mu(\sG/\sG^{(1)})}{\Rk{\sG}}
\end{equation}
and similarly the slope of $\mu(\sF)$ is given by 
\begin{equation}\label{E:slF}
\mu(\sF)=\frac{r_1\mu(\sN\sF)+[r_0-r_1]\mu(\sF^{(1)}/\sN\sF)+[R(\sF)-r_0]\mu(\sF/\sF^{(1)})}{\Rk{\sF}}
\end{equation}
We now conclude that $\mu(\sG)< \mu(\sF)$ by applying Lemma  \ref{Lem:StimaPendenze} to the two expressions \eqref{E:slG} and \eqref{E:slF} and using that 
$$
\rk{\cJ}\leq r_1=\rk{\sN\sF}, \quad \rk{\cG^{(1)}}\leq r_0=\rk{\sF^{(1)}}, \quad \Rk{\sG}\leq \Rk{\sF},
$$
$$
\frac{d_1}{r_1}=\mu(\sN\sF)< \mu(\sF^{(1)}/\sN\sF)=\frac{d_0-d_1-r_1\delta}{r_0-r_1} \quad \text{ because }  \frac{d_1}{r_1}\leq \frac{d_0}{r_0}-\delta<\frac{d_0}{r_0}-\frac{r_1}{r_0}\delta,
$$
$$
\frac{d_1+r_1\delta}{r_1}=\mu(\sF/\sF^{(1)})\leq \mu(\sF^{(1)}/\sN\sF)=\frac{d_0-d_1-r_1\delta}{r_0-r_1} \quad \text{ because } \frac{d_1}{r_1}\leq \frac{d_0}{r_0}-\delta,
$$
$$
r_0(\sG)\Rk{\sF}-r_0\Rk{\sG}=r_1r_0(\sG)-r_0r_1(\sG)< 0,
$$
$$
\Rk{\sG}r_1-\Rk{\sF}\rk{\cJ}\leq \Rk{\sG}r_1-\Rk{\sF}r_1(\sG)=r_1r_0(\sG)-r_0r_1(\sG)< 0,
$$
where in the last two inequalities we have used that  $\rk{\cJ}\ge r_1(\sG)=\rk{\sN\sG}$ since $\sN\sG\subseteq \cJ$ and our assumption $r_1r_0(\sG)< r_0r_1(\sG)$.

\un{Case C:} $(d_0-r_1\delta)/r_0\le d_1/r_1 < d_0/r_0$.

Define the integers 
$$
\begin{sis}
& \wt d_0:=-d_0-r_0\delta+r_1\delta, \\
& \wt d_1:=-d_1-2r_1\delta. 
\end{sis}
$$
It is straightforward to check that the complete type $(r_0,r_1;d_0,d_1)$ verifies the numerical assumption in Case C if and only if $(r_0,r_1;\wt d_0,\wt d_1)$ verifies the numerical assumption in Case B.
Hence, by what proved in Case B, there exists a semistable quasi-locally free $\wt \sF$ of complete type $(r_0,r_1;\wt d_0,\wt d_1)$. The dual sheaf $(\wt\sF)^{\vee} $ is a semistable (by Fact \ref{R:SStab}\ref{R:SStab:2}) quasi-locally free (by Fact \ref{F:Dual}\ref{F:Dual:5}) sheaf of complete type $(r_0,r_1;d_0,d_1)$ by Fact \ref{F:Dual}. 
\end{proof}

\begin{proof}[Proof of the second part of \eqref{T:ss-strata1}] 
This proof proceeds along the same lines of the proof of the second part \eqref{T:ss-strata1}, with the only modification that for $\ov g=1$ we replace Lemma \ref{L:VBsm}\eqref{L:VBsm2} with  Lemma \ref{L:VBsm}\eqref{L:VBsm1}.

\end{proof}

\begin{lemma}\label{L:VBsm}
Let $C$ be a smooth projective irreducible curve of genus $g(C)$.
\begin{enumerate}[(1)]
\item \label{L:VBsm2} Assume that $g(C)\geq 2$. Then
\begin{enumerate}[(a)]
\item  \label{L:VBsm2:1} For any $(r,d)\in \bN_{>0}\times \bZ$, the moduli space of stable vector bundles on $C$ of rank $r>0$ and degree $d\in \bZ$ is non-empty and irreducible.
\item  \label{L:VBsm2:2}  Fix two pairs $(r_1,d_1),(r_2,d_2)\in \bN_{>0}\times \bZ$ such that $\displaystyle \frac{d_1}{r_1}<\frac{d_2}{r_2}$. Then given two general stable vector bundles $E_1$ (resp. $E_2$) of rank
$r_1$ (resp. $r_2$) and degree $d_1$ (resp. $d_2$), the general extension 
$$
0\to E_1\to E\to E_2\to 0
$$   
is a stable vector bundle. 
\item  \label{L:VBsm2:3}  Fix $(r,d)\in \bN_{>0}\times \bZ$. For any general vector bundle $E$ of rank $r$ and degree $d$, and any general point $p\in C$, the general elementary transformation of $E$ at $p$
$$
0\to \ker{\phi}\to E\xrightarrow{\phi} \cO_p\to 0
$$ 
is a stable vector bundle. 
\end{enumerate}
\item \label{L:VBsm1} Assume that $g(C)=1$. Then
\begin{enumerate}[(a)]
\item  \label{L:VBsm1:1} For any $(r,d)\in \bN_{>0}\times \bZ$, the moduli space of semistable vector bundles on $C$ of rank $r>0$ and degree $d\in \bZ$ is non-empty and irreducible.
\item  \label{L:VBsm1:2}  Fix two pairs $(r_1,d_1),(r_2,d_2)\in \bN_{>0}\times \bZ$ such that $\displaystyle \frac{d_1}{r_1}<\frac{d_2}{r_2}$. Then given two semistable vector bundles $E_1$ (resp. $E_2$) of rank $r_1$ (resp. $r_2$) and degree $d_1$ (resp. $d_2$), any general extension 
$$
0\to E_1\to E\to E_2\to 0
$$   
is a semistable vector bundle. 
\item  \label{L:VBsm1:3}  Fix $(r,d)\in \bN_{>0}\times \bZ$. For any semistable vector bundle $E$ of rank $r$ and degree $d$, and any $p\in C$, the general elementary transformation of $E$ at $p$
$$
0\to \ker{\phi}\to E\xrightarrow{\phi} \cO_p\to 0
$$ 
is a semistable vector bundle. 
\end{enumerate}
\end{enumerate}
\end{lemma}
\begin{proof}
This is well-known: we indicate some references. 

Part \eqref{L:VBsm2:1} is classical. 

Part \eqref{L:VBsm2:2} is the so-called Lange's conjecture, see \cite[Prop. 1.11]{RT}.


Part \eqref{L:VBsm1:1} follows from \cite[Thm. 1]{T}.

Part \eqref{L:VBsm1:2} is proved in \cite[Thm. 1]{BR}.

Parts \eqref{L:VBsm2:3} and \eqref{L:VBsm1:3} are well-known, see e.g. \cite[End of the proof of Thm. 0.3]{RT}. 
\end{proof}

\begin{lemma}\label{Lem:VarLemDr}
Let $R_1$, $R_2$, $R_3$, $R_4$, $R_ 5$ and $R_6$ be positive real numbers and let $D_1$, $D_2$, $D_3$, $D_4$, $D_5$ and $D_6$ be real numbers such that
\begin{gather*}
R_1=R_2+R_3,\; R_4=R_5+R_6,\\
D_1=D_2+D_3,\;D_4=D_5+D_6. 
\end{gather*}
Let $\mu_i=R_i/D_i$ for $i=1,\dotsc,6$.
Assume that $\mu_2\ge\mu_3$ (resp. $\mu_5\ge\mu_6$) and that $\mu_6\ge\mu_3$, $\mu_5\ge\mu_2$ and $R_4/R_1\ge R_6/R_3$.
Then it holds that $\mu_4\ge\mu_1$.\\ 
If, moreover, $\mu_2>\mu_3$ (resp. $\mu_5>\mu_6$) or $\mu_6>\mu_3$ or $\mu_5>\mu_2$, then $\mu_4>\mu_1$.
\end{lemma}
\begin{proof}
The case $\mu_2\ge\mu_3$ is proved in \cite[Lemme 5.1.1]{DR3}. 
The (similar) proof in the other case $\mu_5\ge\mu_6$ follows from the following chain of inequalities:
$$R_4R_1(\mu_4-\mu_1)=(R_2+R_3)(D_5+D_6)-(R_5+R_6)(D_2+D_3)\geq
$$
$$
\geq (R_2+R_3)(R_5\mu_5+R_6\mu_6) -(R_5+R_6)(R_2\mu_5+R_3\mu_6)=$$
$$=(R_5R_3-R_2R_6)(\mu_5-\mu_6)\geq \mu_5-\mu_6\geq 0,
$$
where in the first inequality we have used that $\mu_6\ge\mu_3$ and $\mu_5\ge\mu_2$, the second inequality follows from $R_5/R_2\ge R_6/R_3$ which is equivalent to 
$R_4/R_1\ge R_6/R_3$, and the last inequality follows from $\mu_5\ge \mu_6$.
The last inequality  is due to the fact that the first inequality is strict if either $\mu_6>\mu_3$ or $\mu_5>\mu_2$, while the last inequality is strict if $\mu_5>\mu_6$. 
\end{proof} 

\begin{lemma}\label{Lem:StimaPendenze}
Let $m_1>m_2>m_3$ and $m'_1>m'_2>m'_3$ be non-negative integers and let $q_1,q_2,q_3$ and $q'_1, q'_2, q'_3$ be real numbers. Assume $q_1\le q'_1$, $q_2\le q'_2$, $q_3\le q'_3$, $q'_1\le q'_2$ and $q'_3\le q'_2$, $m_1m'_3-m'_1m_3\le 0$ and $m_2m'_1-m'_2m_1\le 0$. Then we have that 
$$w:=\frac{m_3q_1+(m_2-m_3)q_2+(m_1-m_2)q_3}{m_1}\leq \frac{m'_3q'_1+(m'_2-m'_3)q'_2+(m'_1-m'_2)q'_3}{m'_1}:=w'$$ 
Moreover, if either $q_1'<q_2'$ and $m_1m'_3-m'_1m_3< 0$ or $q'_3<q'_2$ and  $m_2m'_1-m'_2m_1< 0$, then $w<w'$.
\end{lemma}

\begin{proof}
We compute 
$$w'-w=\frac{m'_3}{m'_1}q'_1-\frac{m_3}{m_1}q_1 + \frac{m'_2-m'_3}{m'_1}q'_2-\frac{m_2-m_3}{m_1}q_2+\frac{m'_1-m'_2}{m'_1}q'_3-\frac{m_1-m_2}{m_1}q_3\ge$$
$$\ge \frac{q'_1(m_1m'_3-m'_1m_3)+q'_2(m'_2m_1-m'_3m_1-m_2m'_1+m_3m'_1)+q'_3(m'_1m_1-m'_2m_1-m_1m'_1+m_2m'_1)}{m'_1m_1}=$$
$$=\!\frac{(m_1m'_3-m'_1m_3)(q'_1-q'_2)+(m_2m'_1-m'_2m_1)(q'_3-q'_2)}{m'_1m_1} \ge 0.$$
 The last assertion is clear from the above last inequality. 
\end{proof}

\begin{rmk}\label{R:ss-strata}
\noindent

\begin{enumerate}

\item\label{R:ss-strata1}
Using \eqref{E:d->b}, Theorem \ref{T:ss-strata} for generalized vector bundles can be reformulated using the index as it follows: if $\sF$ is a semistable (resp. stable) generalized vector bundle with $R(\sF)=2r$ then $\iota(\sF)\le r\delta$ (resp. $\iota(\sF)<r\delta$) and such semistable (resp. stable) generalized vector bundles exist if $\ov g\geq 1$ (resp. $\ov g\geq 2$).

In the case of generalized line bundles, more is true: any generalized line bundle $\sF$ is semistable (resp. stable) if and only if $\iota(\sF) \le\delta$ (resp. $\iota(\sF)<\delta$), and such semistable (resp. stable) generalized line bundles always exist (see \cite[Thm. A]{CK}).

\item \label{R:ss-strata2}

The second part of Theorem \ref{T:ss-strata}\eqref{T:ss-strata1} (resp. \ref{T:ss-strata}\eqref{T:ss-strata2}) is false, for arbitrary $(r_\bullet;d_\bullet)$,  for $\ov g=0$ (resp. $\ov g=0,1$). 
For example:
\begin{itemize}
\item If $\ov g=0$ then 
$$\begin{sis}
& \cM_X(R,0;D,0)\cap \cM_X(R,D)^{ss}\neq \emptyset \text{ if and only if } R\vert D, \\
& \cM_X(R,0;D,0)\cap \cM_X(R,D)^{s}\neq \emptyset \text{ if and only if } R=1.
\end{sis}$$
\item If $\ov g=1$ then $\cM_X(R,0;D,0)\cap \cM_X(R,D)^{s}\neq \emptyset$  if and only if  $R$ and $D$ are coprime (see \cite[p. 3]{T}).
\end{itemize}

\end{enumerate}

\end{rmk}

\begin{cor}\label{C:neg-delta}(see \cite[\S 1.2]{DR3})
If $\delta <0$ (resp. $\delta \leq 0$) then $\cM_X(R,D)^{ss}\subseteq \cM_{X}(R,0;D,0)$ (resp. $\cM_X(R,D)^{s}\subseteq \cM_{X}(R,0;D,0)$). 
\end{cor}
\begin{proof}
This follows from the previous Theorem together with the fact that if $(r_0,r_0;d_0,d_1)$ is admissible then $d_1+r\delta\leq d_0$: indeed, if $\delta <0$ (resp. $\delta \leq 0$) then the inequality in Theorem \ref{T:ss-strata}\eqref{T:ss-strata1} (resp. Theorem \ref{T:ss-strata}\ref{T:ss-strata2}) are satisfied if and only if $r_1=0$. 
\end{proof}

\subsection{Example: semistability of torsion-free sheaves of type $(n,1)$}\label{sub:ss(n,1)}

In this subsection, we discuss the semistability of torsion-free sheaves of type $(n,1)$ (for $n\geq 0$), i.e.   those torsion-free sheaves $\sF$ on $X$ such that 
$$\sF_\eta\cong\cO_{\Xred,\eta}^{\oplus n}\oplus\cO_{X,\eta}, \: \text{ or, equivalently, such that } (r_0(\sF),r_1(\sF))=(n+1,1).$$

First of all, we show that these sheaves are the push-forward of quasi-locally free sheaves of the same type on some blow-up of the ribbon.


\begin{prop}\label{P:BlowUp}
Let $\sF$ be a torsion-free sheaf on $X$ of type $(n,1)$, for some $n\geq 0$. There is a unique divisor $D\subset \Xred$ such that $\cT(\sF|_{\Xred})$ is isomorphic to $\cO_D$ and a unique quasi-locally free sheaf $\sF'$ of type $(n,1)$ on the blow up $q:\Bl_D(X)\to X$ of $X$ at $D$ such that $q_*\sF'\simeq\sF$.
\end{prop}
This theorem generalizes \cite[Theorem 1.1]{EG} which deals with generalized line bundles (i.e. with the case $n=0$).
\begin{proof}
The proof is essentially the same as the proof of \cite[Theorem 1.1]{EG}, as we now indicate. 

First of all, observe that  $\sN\sF$ and its saturation $\sK:=\ker{\sF\to \frac{\sF_{|\Xred}}{\cT(\sF|_{\Xred})}}$ are line bundles on $\Xred$ (being torsion-free of rank $r_1(\sF)=1$ by Remark \ref{R:type}) and that 
$\sK/\sN\sF\cong \cT(\sF|_{\Xred})$ can be written as $\cO_D$ for a unique effective divisor $D$ of $\Xred$.

At this point the proof is \emph{verbatim} the same of \cite[Theorem 1.1]{EG}: it is possible to give to $\sF$ a structure of $\cO_{\Bl_D(X)}$-module (which is unique because it is derived only from the $\cO_X$-module structure of $\sF$) in the following way. 

Let $f\in \operatorname{H}^0(\cO_{\Xred}(D))$ be a section vanishing on $D$, let $\sigma'$ be a section of $\cO_{\Bl_D(X)}$ defined over an open set $U$ of $X$ (recall that $X$ and $\Bl_D(X)$ are homeomorphic) and $m$ a section of $\sF(U)$. Shrinking $U$, if necessary, it is possible to find a section $\sigma$ of $\cO_X(U)$ with the same image of $\sigma'$ in $\cO_{\Xred}(U)$. Hence, $\sigma'=\sigma+f^{-1}\tau$, where $\tau$ is an appropriate section of $\sN(U)$. The sheaf $\sF$ admits a structure of $\cO_{\Bl_D(X)}$-module by setting $\sigma'm:=\sigma m+ f^{-1}(\tau m)$; the latter is well defined because $\tau m\in\sN\sF$ and $\sK=\cO_{\Xred}(\sF)\otimes\sN\sF$ and it is independent of the choice of $\sigma$. By construction, we have that $q_*\sF'\simeq\sF$.
\end{proof}


For torsion-free sheaves of type $(n,1)$, we have the following strong (semi)stability criterion.

\begin{prop}\label{P:ss(n,1)}
Let $\sF$ be a torsion-free sheaf on $X$ of type $(n,1)$, for some $n>0$. Then $\sF$ is semistable if and only if the following two conditions are verified:
\begin{enumerate}
\item\label{P:ss(n,1)1} for any sub-bundle $\cE\subseteq\sF^{(1)}$ of rank less than $n+1$ we have $\mu(\cE)\le\mu(\sF)$;
\item\label{P:ss(n,1)2} for any torsion-free quotient $(\sF|_{\Xred})^{\vee\vee}\surj \cG$ of rank less than $n+1$ it holds that $\mu(\cG)\ge\mu(\sF)$. 
\end{enumerate}  
Furthermore, $\sF$ is stable if and only if the inequalities in \ref{P:ss(n,1)1} and \ref{P:ss(n,1)2} are strict.
\end{prop}
The above Proposition extends \cite[Lemme 9.1.2]{DR1}, which deals with quasi-locally free sheaves of type $(1,1)$. 
\begin{proof}
Necessity is obvious, so we have to only prove sufficiency. We will prove only the semistable case, because the stable one is similar. 

Let $\sE\subset \sF$ be a saturated subsheaf. If $\sE$ is defined over $\Xred$, then $\sE\subseteq\sF^{(1)}$. If it has rank $\le n$, then $\mu(\sE)\le\mu(\sF)$ by \ref{P:ss(n,1)1}. On the other hand, if it has rank $n+1$, it has the same rank of $\sF^{(1)}$ and it is contained in it. Hence, $\mu(\sE)\le\mu(\sF^{(1)})$ and it suffices to check that $\mu(\sF^{(1)})\le\mu(\sF)$. This follows from condition \ref{P:ss(n,1)2}, because $\sF/\sF^{(1)}$ is a torsion-free quotient of $(\sF|_{\Xred})^{\vee\vee}$ of rank $1$ and, thus, $\mu(\sF/\sF^{(1)})\ge\mu(\sF)$.  

So, assume that $\sE$ is not defined over $\Xred$; this means that $\sE$ is generically isomorphic to $\cO_X\oplus \cO_{\Xred}^{\oplus m}$ with $0\le m<n$. Hence, $\sF/\sE$ is generically isomorphic to $\cO_{\Xred}^{\oplus(n-m)}$. This implies (using that $\sE$ is saturated in $\sF$) that $\sF/\sE$ is a rank $n-m$ vector bundle on $\Xred$. Thus, $\sF/\sE$ is a pure quotient of $(\sF|_{\Xred})^{\vee\vee}$ of rank $\le n$ and so, by \ref{P:ss(n,1)2}, $\mu(\sF/\sE)\ge\mu(\sF)$, which is equivalent to $\mu(\sE)\le\mu(\sF)$.   
\end{proof} 

\begin{rmk}\label{R:ss(0,1)}
The case $n=0$, i.e. that of generalized line bundles, is not covered by the previous Proposition. In order to cover also this case, one should drop the hypothesis \emph{of rank $\le n$} in the two conditions, which moreover becomes equivalent to each other. Indeed if $\sI$ is a generalized line bundle, it holds that $(\sI|_{\Xred})^{\vee\vee}=\sI/\sI^{(1)}$ and the two conditions (without the cited hypothesis) are both equivalent to $b(\sI)\le-\deg(\sN)$, which is equivalent to the semistability of $\sI$ (see \cite[Lemma 3.2]{CK}).
\end{rmk}

\begin{cor}\label{C:ss(n,1)}
Let $\sF$ be a torsion-free sheaf of type $(n,1)$, for some $n\geq 0$. Let $q:\Bl_D(X)\to X$ be unique the blow up of $X$ at a divisor $D\subset \Xred$ such that $\sF=q_*(\sF')$ for a quasi-locally free sheaf $\sF'$ on $X'$, as in Proposition \ref{P:BlowUp}. Then $\sF$ is (semi)stable if and only if $\sF'$ is (semi)stable. 
\end{cor}
\begin{proof}
It holds by definition that $\Deg{\sF}=\Deg{\sF'}$ and $\Rk{\sF}=\Rk{\sF'}$; hence, $\mu(\sF)=\mu(\sF')$. The construction of $\sF'$ implies that $\sF^{(1)}=\sF'^{(1)}$ and $(\sF|_{\Xred})^{\vee\vee}=(\sF'|_{\Xred})^{\vee\vee}=\sF'|_{\Xred}$. Therefore, the assertion follows from Proposition \ref{P:ss(n,1)} for $n>0$  and from Remark \ref{R:ss(0,1)} for $n=0$.
\end{proof}

\section{Tangent spaces}\label{Sec:tangent}

The aim of this section is to compute the infinitesimal automorphism space $\Inf_{\sF}\cM_X(R,D)$ and the tangent space $T_{\sF}\cM_X(R,D)$ of $\cM_X(R,D)$ at a point $\sF$ (see \cite[\href{https://stacks.math.columbia.edu/tag/07WY}{Tag 07WY}]{stacks}). Since $\cM_X(R,D)$ is an open substack of the stack of coherent sheaves on $X$, we have the following identifications (see e.g. the proof of  
\cite[\href{https://stacks.math.columbia.edu/tag/08W8}{Tag 08W8}]{stacks})
\begin{equation}\label{E:Inf-T}
\Inf_{\sF}\cM_X(R,D)=\Ext{0}(\sF,\sF) \quad \text{ and } \quad T_{\sF}\cM_X(R,D)=\Ext{1}(\sF,\sF).
\end{equation}
Using the local-to-global spectral sequence 
$$
E_2^{p,q}=H^p(X,\cExt{q}(\sF,\sF))\Rightarrow \Ext{p+q}(\sF,\sF),
$$
and the fact that $X$ is a curve, we deduce the following identification and short exact sequence:
\begin{equation}\label{E:Ext}
\begin{sis}
& \Ext{0}(\sF,\sF)=H^0(X, \cExt{0}(\sF,\sF)), \\ 
& 0\to H^1(X,\cExt{0}(\sF,\sF)) \to \Ext{1}(\sF,\sF) \to H^0(X,\cExt{1}(\sF,\sF)) \to 0.
\end{sis}
\end{equation}

\begin{rmk}\label{R:T-strata}
In the  exact sequence of \eqref{E:Ext}, the first term $H^1(X,\cExt{0}(\sF,\sF))$ parametrizes infinitesimal deformations of $\sF$ that are locally trivial, i.e. such that the induced infinitesimal deformation of $\sF_p$ is trivial for each point $p\in X$, and hence they do not change  the complete type of $\sF$. Therefore we get the inclusions 
\begin{equation}\label{E:inc-T}
H^1(X,\cExt{0}(\sF,\sF))\subseteq T_{\sF}\cM_X(r_\bullet;d_\bullet)\subseteq   T_{\sF}\cM_X(R,D)=\Ext{1}(\sF,\sF).
\end{equation}
On the other hand, we clearly have the equalities
\begin{equation}\label{E:inc-Inf}
H^0(X,\cExt{0}(\sF,\sF)) = \Inf_{\sF}\cM_X(r_\bullet;d_\bullet)=  \Inf_{\sF}\cM_X(R,D)=\Ext{0}(\sF,\sF).
\end{equation}
\end{rmk}

We now study the sheaves $\cExt{0}(\sF,\sF)$ and $\cExt{k}(\sF,\sF)$ for $k\geq 1$.

\begin{thm}\label{T:Ext0}
Let $\sF\in \cM_X(R,D)$ and set $\cT:=\cT(\sF_{|\Xred})$. Then $\cExt{0}(\sF,\sF)=\cEnd(\sF)$ is a torsion-free sheaf on $X$ that fits into an exact sequence
\begin{equation}\label{E:seqExt0}
0\to \cHom((\sF_{|\Xred})^{**}, \sF^{(1)})\xrightarrow{\alpha} \cEnd(\sF)\xrightarrow{\beta} \cEnd(\sN \sF)\to \bigoplus_{p\in \supp(\cT)}
 \bigoplus_{1\leq i<j\leq b(\sF)}\cO_{(n_i(\sF,p)-n_j(\sF,p))p}\to 0.
\end{equation}
Moreover, the short exact sequence 
\begin{equation}\label{E:seqExtbis}
0\to \cHom((\sF_{|\Xred})^{**}, \sF^{(1)})\xrightarrow{\alpha} \cEnd(\sF)\xrightarrow{\beta}  \Im(\beta) \to 0
\end{equation}
is the second canonical filtration of $\cEnd(\sF)$. 
\end{thm}
Note that the cokernel of $\beta$ is a torsion sheaf supported on $\supp(\cT(\sF_{|\Xred}))$ and whose length is equal to (by Definition \ref{Def:index} and Remark \ref{Rmk:Def:index})
\begin{equation}\label{E:cokerbeta}
\begin{aligned}
&l(\coker \beta)=\sum_{p\in \supp(\cT)}\sum_{1\leq i<j\leq b(\sF)} [(n_i(\sF,p)-n_j(\sF,p)]=\\
&=\sum_{p\in \supp(\cT)}\sum_{1\leq i\leq b(\sF)} (b(\sF)-2i+1) n_i(\sF,p)=b(\sF)\iota(\sF)-\sum_{p\in \supp(\cT)}\sum_{1\leq i\leq b(\sF)} (2i-1) n_i(\sF,p).
\end{aligned}
\end{equation}
The exact sequence \eqref{E:seqExt0} is proved for quasi-locally free sheaves (in which case $\beta$ is surjective) in \cite[Prop. 3.12]{DR2}.
\begin{proof}
The fact that $\cEnd(\sF)$ is a torsion-free sheaf follows from the fact that $\sF$ is a torsion-free sheaf.

Let us first define the maps $\alpha$ and $\beta$. Observe that  the natural morphism $\cHom((\sF_{|\Xred})^{**}, \sF^{(1)})\to \cHom(\sF_{|\Xred}, \sF^{(1)})$ is an isomorphism, since  $\sF^{(1)}$ is a torsion-free sheaf on $\Xred$ and hence any morphism $\phi:\sF_{|\Xred}\to \sF^{(1)}$ factors through the torsion-free quotient $\sF_{|\Xred}\twoheadrightarrow (\sF_{|\Xred})^{**}$. The morphism $\alpha$ is defined by 
$$
\begin{aligned}
\alpha: \cHom((\sF_{|\Xred})^{**}, \sF^{(1)})=\cHom(\sF_{|\Xred}, \sF^{(1)})& \longrightarrow  \cEnd(\sF)\\
\phi:\sF_{|\Xred}\to \sF^{(1)} &\mapsto \alpha(\phi):\sF\twoheadrightarrow  \sF_{|\Xred}\xrightarrow{\phi}  \sF^{(1)}\hookrightarrow \sF.
\end{aligned}
$$
The map $\beta$ is the restriction to $\sN\sF$ 
$$
\begin{aligned}
\beta:  \cEnd(\sF) & \longrightarrow \cEnd(\sN \sF)\\
\Psi & \mapsto \beta(\Psi):=\Psi_{|\sN\sF} 
\end{aligned}
$$ 
which is well-defined (i.e $\Psi(\sN\sF)\subseteq \sN\sF$) because $\Psi$ is a homomorphism of $\cO_X$-modules and $\sN$ is an ideal sheaf.

Let us now show that the sequence \eqref{E:seqExt0} is exact. The map $\alpha$ is obviously injective and its image is identified with the endomorphisms of $\sF$ which factors through the restriction $\sF\twoheadrightarrow \sF_{\Xred}$, since such endomorphisms have image supported on $\Xred$ and hence contained in $\sF^{(1)}$. Since an endomorphism of $\sF$ factors through the restriction $\sF\twoheadrightarrow \sF_{\Xred}$ if and only if its restriction to $\sN\sF$ is zero, we conclude that $\Im(\alpha)=\ker{\beta}$. In order to describe the cokernel of  $\beta$, fix $p\in X$ and write 
$$
\sF_p= \cO_{\Xred,p}^{\oplus a(\sF)}\oplus  \bigoplus_{i=1}^{b(\sF)} \cI_{n_ip},
$$
with $n_\bullet(\sF,p):=\{n_1=n_1(\sF,p)\geq \ldots \geq n_b=n_b(\sF,p)\}$, as in Fact \ref{F:structure}. Therefore, we have that 
$$
\coker{\beta}_p= \bigoplus_{1\leq i,j\leq b(\sF)}\coker{\Hom(\wh \cI_{n_ip}, \wh \cI_{m_ip}) \xrightarrow{\res} \Hom(\wh \cN_p \wh \cI_{n_ip}, \wh \cN_p \wh \cI_{m_ip})}.
$$
We now conclude using Lemma \ref{L:locHom} below.

It remains to show that \eqref{E:seqExtbis} is the second canonical filtration of $\cEnd(\sF)$. This follows from the fact that an endomorphism of $\sF$ is annihilated by $\sN$ if and only if it factors through the restriction $\sF\twoheadrightarrow \sF_{|\Xred}$, i.e. if and only it belongs to $\Im(\alpha)$. 
\end{proof}

\begin{lemma}\label{L:locHom}
Let $\wh \cO_{X,p}$ be the completion of the local ring of a ribbon $X$ at a closed point $p$. Denote by $\wh \cN_p$ the nilradical of $\wh \cO_{X,p}$ and, for any $n\geq 0$, consider the ideal $\wh \cI_{np}$ of $np$ inside $\wh \cO_{X,p}$ and the structure sheaf $\wh \cO_{np}$ of $np$. Then for any $n,m\geq 0$, the cokernel of the restriction map
$$
\res: \Hom(\wh \cI_{np}, \wh \cI_{mp}) \to \Hom(\wh \cN_p \wh \cI_{np}, \wh \cN_p \wh \cI_{mp}) 
$$
is isomorphic to $\wh \cO_{\max\{n-m,0\}p}$. 
\end{lemma}
\begin{proof}
First of all, observe that the rings $ \wh \cO_{X,p}$ and the $\wh \cO_{X,p}$-modules $\wh \cN_p$, $\wh \cO_{\Xred,p}$ and $\wh \cO_{np}$ can be written as 
\begin{equation}\label{E:pres-mod}
\begin{sis}
& \wh \cO_{X,p}=k[[s,\epsilon]]/(\epsilon^2),\\
&\wh \cN_p=(\epsilon) \subset  \wh \cO_{X,p} \Rightarrow  \wh \cO_{\Xred,p}=\wh \cO_{X,p}/\wh \cN_p=k[[s]],\\
& \wh \cO_{np}=\frac{\wh \cO_{X,p}}{(\epsilon, s^n)}\cong \frac{k[[s]]}{(s^n)}.\\
\end{sis}
\end{equation}
On the other hand, the ideal $\wh \cI_{np}$ admits the following presentation as a $ \wh \cO_{X,p}$-module:
\begin{equation}\label{E:pres-Inp}
\begin{aligned}
\wh \cI_{np}=(\epsilon, s^n) & \xrightarrow{\cong} \frac{ \wh \cO_{X,p}e\oplus  \wh \cO_{X,p}f}{(\epsilon f, \epsilon e-s^nf)}\\
\epsilon & \mapsto f\\
s^n & \mapsto e,
\end{aligned}
\end{equation}
which implies, in particular, that any element of $\wh \cI_{np}$ can be written uniquely as $a(s)e+b(s)f$, where $a(s),b(s)\in k[[s]]$. Therefore, the module $\wh \cN_p\wh \cI_{np}$ is isomorphic to 
$$
\begin{aligned}
\wh \cO_{\Xred,p}=k[[s]]&  \xrightarrow{\cong}\wh \cN_p\wh \cI_{np}=\langle\epsilon e, \epsilon f\rangle\\
1 & \mapsto \epsilon e=s^n f.
\end{aligned}
$$

We now distinguish two cases:

$\bullet$ If $n\leq m$ then any element $\phi\in \Hom(\wh \cI_{np}, \wh \cI_{mp})$ can be written as
$$
\phi(e)=a_1(s) e+b_1(f) \quad \text{ and } \quad \phi(f)=s^{m-n}a_1(s) f \quad \text{ for some } a_1(s),b_1(s)\in k[[s]].
$$
The restriction of $\phi$ to $\Hom(\wh \cN_p \wh \cI_{np}, \wh \cN_p \wh \cI_{mp}) $ is given by 
$$
\res(\phi)(s^nf)=a_1(s) s^mf, 
$$
which shows that $\res$ is surjective. 

$\bullet$ If $n\geq m$ then any element $\phi\in \Hom(\wh \cI_{np}, \wh \cI_{mp})$ can be written as
$$
\phi(e)=s^{n-m} b_2(s) e+b_1(f) \quad \text{ and } \quad \phi(f)=b_2(s) f \quad \text{ for some } b_1(s),b_2(s)\in k[[s]].
$$
The restriction of $\phi$ to $\Hom(\wh \cN_p \wh \cI_{np}, \wh \cN_p \wh \cI_{mp}) $ is given by 
$$
\res(\phi)(s^nf)=b_2(s) s^nf=b_2(s) s^{n-m} (s^m f). 
$$
which shows that the image of $\res$ is equal to 
$$s^{n-m} \Hom(\wh \cN_p \wh \cI_{np}, \wh \cN_p \wh \cI_{mp})\subset  \Hom(\wh \cN_p \wh \cI_{np}, \wh \cN_p \wh \cI_{mp})=\wh\cO_{\Xred, p},$$
and hence that the cokernel of $\res$ is isomorphic to $\wh \cO_{(n-m)p}$.

\end{proof}

\begin{cor}\label{C:Ext0}
Let $\sF\in \cM_X(r_\bullet;d_\bullet)$ and set $\cT:=\cT(\sF_{|\Xred})$. 
\begin{enumerate}
\item  \label{C:Ext0-1} The complete type of $\cEnd(\sF)$ is equal to 
$$
\begin{sis}
& (r_0(\cEnd(\sF)), r_1(\cEnd(\sF)))=(r_0^2,r_1^2),\\
& d_0(\cEnd(\sF))=  -\delta r_1(r_0-r_1)+r_0\iota(\sF), \\
& d_1(\cEnd(\sF))=   -\delta r_1^2-b(\sF)\iota(\sF)+\sum_{p\in \supp(\cT)}\sum_{1\leq i\leq b(\sF)} (2i-1) n_i(\sF,p).
\end{sis}
$$
\item \label{C:Ext0-2} The generalized rank and degree of $\cEnd(\sF)$ are equal to 
$$
\begin{sis}
& R(\cEnd(\sF))=r_0^2+r_1^2,\\
& D(\cEnd(\sF))=  -\delta r_1r_0+a(\sF) \iota(\sF)+\sum_{p\in \supp(\cT)}\sum_{1\leq i\leq b(\sF)} (2i-1) n_i(\sF,p).
\end{sis}
$$
\item \label{C:Ext0-3} The Euler characteristic of $\cEnd(\sF)$ is equal to
$$
\begin{aligned}
& \chi(\cEnd(\sF))=  -\delta r_1r_0+a(\sF) \iota(\sF)+\sum_{p\in \supp(\cT)}\sum_{1\leq i\leq b(\sF)} (2i-1) n_i(\sF,p)+(r_0^2+r_1^2)(1-\ov g)=\\
& =-\dim \cM_X(r_\bullet; d_\bullet)+a(\sF)\iota(\sF)+\sum_{p\in \supp(\cT)}\sum_{1\leq i\leq b(\sF)} (2i-1) n_i(\sF,p).
\end{aligned}
$$
\end{enumerate}
\end{cor}
\begin{proof}
Part \ref{C:Ext0-1} follows from Theorem \ref{T:Ext0} together with Remark \ref{R:inv-F1} and formula \eqref{E:cokerbeta}.

Part \ref{C:Ext0-2} follows from part \ref{C:Ext0-1} together with Definition \ref{D:GenRkDeg} and Remark \ref{R:GenRkDeg}.

Part \ref{C:Ext0-3} follows from part \ref{C:Ext0-2} together with the Riemann-Roch formula (see Fact \ref{F:GenRkDeg}\ref{F:GenRkDeg:4}) and Theorem \ref{T:irr-dim}.
\end{proof}

\begin{cor}\label{C:Tan-lt}
Let $\sF\in \cM_X(r_\bullet;d_\bullet)$ and set $\cT:=\cT(\sF_{|\Xred})$. The inclusion 
$$
H^1(X,\cEnd(\sF))\subseteq \dim T_{\sF}\cM_X(r_\bullet;d_\bullet)
$$
has codimension equal to 
$$
a(\sF)\iota(\sF)+\sum_{p\in \supp(\cT)}\sum_{1\leq i\leq b(\sF)} (2i-1) n_i(\sF,p).
$$
In particular:
\begin{enumerate}
\item \label{C:Tan-lt1} $\sF$ is quasi-locally free if and only if  $H^1(X,\cExt{0}(\sF,\sF))=T_{\sF}\cM_X(r_\bullet;d_\bullet)$.
\item \label{C:Tan-lt2} If $\sF$ is a generalized vector bundle of rank $r$ then the codimension of the inclusion $H^1(X,\cEnd(\sF))\subseteq \dim T_{\sF}\cM_X(r_\bullet;d_\bullet)$ is equal to 
$$
\sum_{p\in \supp(\cT)}\sum_{1\leq i\leq r} (2i-1) n_i(\sF,p).
$$
Furthermore, if $\sF$ is a general element of $\cM_X(r,r;d_\bullet)$ then the codimension of the inclusion $H^1(X,\cEnd(\sF))\subseteq \dim T_{\sF}\cM_X(r_\bullet;d_\bullet)$ is equal to 
$\iota(\sF)$.
\end{enumerate}
\end{cor}
The only if part of \ref{C:Tan-lt1} follows from \cite[Cor. 6.7]{DR2} (using Remark \ref{R:T-strata}).
\begin{proof}
The main assertion follows from Corollary \ref{C:Ext0}\ref{C:Ext0-3} together with the fact that, since $\cM_X(r_\bullet;d_\bullet)$ is smooth by Theorem \ref{T:irr-dim}, we have that 
$$
\begin{aligned}
& \dim \cM_X(r_\bullet;d_\bullet)=\dim T_{\sF}\cM_X(r_\bullet;d_\bullet)-\dim \Inf_{\sF}\cM_X(r_\bullet;d_\bullet)=\\
&=\dim T_{\sF}\cM_X(r_\bullet;d_\bullet)-\dim H^0(X,\cEnd(\sF)),
\end{aligned}
$$
where we used \eqref{E:inc-Inf} in the last equality.

Parts \ref{C:Tan-lt1} and  \ref{C:Tan-lt1}  follows from the main assertion together with the facts that
\begin{itemize}
\item $\sF$ is quasi-locally free if and only if $\iota(\sF)=0$ and $\cT=0$.
\item If $\sF$ is a generalized vector bundle or rank $r$, then $a(\sF)=0$ and $b(\sF)=r$.
\item If $\sF$ is a general element of $\cM_X(r,r;d_\bullet)$, then we have that 
$$
n_\bullet(\sF,p)=\{n_1(\sF,p)=1>n_2(\sF,p)=\ldots=n_r(\sF,p)=0\} \quad  \text{ for any } p\in \supp \cT(\sF_{|\Xred}),
$$
as it follows from Corollary \ref{C:gengvb}. 
\end{itemize}
\end{proof}

\begin{thm}\label{T:Ext1}
Let $\sF\in \cM_X(R,D)$. Then $\cExt{k}(\sF,\sF)$, for any $k\geq 1$, is a coherent sheaf on $\Xred$ whose torsion subsheaf and torsion-free quotient are, respectively (using the notation of Fact \ref{F:structure}):
$$
\begin{sis}
& \cT(\cExt{k}(\sF, \sF))=\bigoplus_{p\in \supp(\cT(\sF_{|\Xred}))} \bigoplus_{i=1}^{b(\sF)}\cO_{n_i(\sF,p)p}^{\oplus 2(2i-1+a(\sF))}. \\
&  \cExt{k}(\sF, \sF)^{**}= \cEnd\left(\left(\frac{\sF^{(1)}}{\sN\sF}\right)^{**}\right)\otimes \sN^{-1}.
\end{sis}
$$
In particular, we have that:
\begin{enumerate}[(i)]
\item the torsion subsheaf  $\cT(\cExt{k}(\sF, \sF))$ has length equal to 
$$2\iota(\sF)a(\sF)+2\sum_{p\in \supp(\cT(\sF_{|\Xred}))}(2i-1)n_i(\sF,p)\geq 2\iota(\sF)[a(\sF)+1].$$
\item the torsion-free quotient $\cExt{k}(\sF, \sF)^{**}$ has  rank  $(r_0(\sF)-r_1(\sF))^2$ and slope $\delta$.
\end{enumerate}
\end{thm}
This was proved for 
 quasi locally free sheaves in \cite[Cor. 6.2.2]{DR1} and for generalized line bundles in \cite[Lemma 4.12]{CK} (see also the proof of \cite[Prop. 8.1.3]{DR1}).
\begin{proof}
Fix $p\in X$ and write 
\begin{equation}\label{E:stalkF}
\sF_p= \cO_{\Xred,p}^{\oplus a(\sF)}\oplus  \bigoplus_{i=1}^{b(\sF)} \cI_{n_ip},
\end{equation}
with $n_\bullet(\sF,p):=\{n_1=n_1(\sF,p)\geq \ldots \geq n_b=n_b(\sF,p)\}$, as in Fact \ref{F:structure}. Lemma \ref{L:locExt} implies that, for any $k\geq 2$, 
$$
\cExt{k}(\sF, \sF)_p=\cO_{\Xred,p}^{\oplus a(\sF)^2}\bigoplus_{i=1}^{b(\sF)} \cO_{n_ip}^{2a(\sF)}\bigoplus_{i=1}^{b(\sF)}\cO_{n_ip}^{2(2i-1)}.
$$
This implies that $\cExt{k}(\sF,\sF)$ is supported on $\Xred$ (which is also proved in \cite[Thm. 5.6]{DR2}), that the torsion subsheaf $\cT(\cExt{k}(\sF, \sF))$ is given by the formula in the Theorem and that the torsion-free quotient  $\cExt{k}(\sF, \sF)^{**}$ is a vector bundle on $\Xred$ of rank $a(\sF)^2$.

In order to compute $\cExt{k}(\sF, \sF)^{**}$, we apply Fact \ref{Fact:Nonqll} to get an exact sequence 
\begin{equation}\label{E:seqEF}
0\to \sE \xrightarrow{p} \sF \xrightarrow{q} \cT:=\cT(\sF_{|\Xred})\to 0
\end{equation}
with the property that $\sE$ is quasi-locally free and $\sN\sE=\sN\sF$. By comparing the canonical exact sequences \eqref{Eq:ExSeq1} of $\sE$ and $\sF$, we deduce that 
\begin{equation}\label{E:equaz1}
\frac{\sF^{(1)}}{\sN\sF}\cong\frac{\sE^{(1)}}{\sN\sE}\oplus \cT(\sF_{|\Xred}) \Rightarrow
\left(\frac{\sF^{(1)}}{\sN\sF}\right)^{**}\cong \frac{\sE^{(1)}}{\sN\sE}.
\end{equation}
We can now apply \cite[Cor. 6.2.2]{DR1} to the quasi-locally free sheaf $\sE$ in order to get that 
\begin{equation}\label{E:equaz2}
\cExt{k}(\sE, \sE)= \cEnd\left(\frac{\sE^{(1)}}{\sN\sE}\right)\otimes \sN^{-1} \quad \text{ for any } k\geq 1.
\end{equation}
Using \eqref{E:equaz1} and \eqref{E:equaz2}, it remains to prove  that 
\begin{equation}\label{E:equaz3}
\cExt{k}(\sF, \sF)^{**}\cong \cExt{k}(\sE, \sE) \quad \text{ for any } k\geq 1.
\end{equation}
This will follow from the next two Claims.

\vspace{0.1cm}

\un{Claim 1:} We have an isomorphism $p^*:\cExt{k}(\sF, \sF)^{**}\xrightarrow{\cong} \cExt{k}(\sE, \sF)^{**}$. 

\vspace{0.1cm}

Indeed, by taking the stalk of \eqref{E:seqEF} at a given $p\in X$ and using \eqref{E:stalkF}, we deduce the exact sequence
\begin{equation}\label{E:seqstalk}
0\to \sE_p= \cO_{\Xred,p}^{\oplus a(\sF)}\oplus  \cO_{X,p}^{\oplus b(\sF)}  \to \sF_p= \cO_{\Xred,p}^{\oplus a(\sF)}\oplus  \bigoplus_{i=1}^{b(\sF)} \cI_{n_ip}\to  \cT_p=\bigoplus_{i=1}^{b(\sF)} \cO_{n_ip} \to 0
\end{equation}
Using the exact sequence $0\to   \cO_{X,p}\to  \cI_{n_ip}\to \cO_{n_ip}\to 0$
and the vanishing $\Ext{k}(\cO_{X,p}, M)=0$ for any $k\geq 1$ and any module $M$ over $\cO_{X,p}$, we get the isomorphism
$$
\Ext{k+1}(\cO_{n_ip}, M)\xrightarrow{\cong} \Ext{k+1}(\cI_{n_ip}, M) \quad \text{ for any } k\geq 1 \text{ and any $\cO_{X,p}$-module } M.  
$$
Applying this isomorphism to the exact sequence \eqref{E:seqstalk}, we deduce that 
$$
\Ext{k+1}(\cT_p, M)\hookrightarrow \Ext{k+1}(\sF_p, M) \quad \text{ is injective for any } k\geq 1 \text{ and any $\cO_{X,p}$-module } M.  
$$
This implies that 
\begin{equation}\label{E:injExt}
\cExt{k+1}(\cT, \sF)\stackrel{q^*}{\hookrightarrow} \cExt{k+1}(\sF, \sF) \quad \text{ is injective for any } k\geq 1.
\end{equation}
From the long exact sequence associated to the short exact sequence  \eqref{E:seqEF} and the injectivity \eqref{E:injExt}, we deduce an exact sequence
\begin{equation}\label{E:rexExt}
\ldots \to \cExt{k}(\cT, \sF)\xrightarrow{q^*}\cExt{k}(\sF, \sF)\xrightarrow{p^*} \cExt{k}(\sE, \sF)\rightarrow 0 \quad \text{ for any } k\geq 1. 
\end{equation}
Since $\cExt{k}(\cT, \sF)$ is a torsion sheaf (being supported on $\supp(\cT)$), by taking the reflexive hull of \eqref{E:rexExt} we get the desired isomorphism of Claim 2.

\vspace{0.1cm}

\un{Claim 2:} We have an isomorphism $p_*:\cExt{k}(\sE, \sE)\xrightarrow{\cong} \cExt{k}(\sE, \sF)^{**}$. 

\vspace{0.1cm}

Indeed, from the long exact sequence associated to the short exact sequence  \eqref{E:seqEF}, and using that the connecting homomorphisms $\cExt{k-1}(\sE,\cT)\to \cExt{k}(\sE,\sE)$ are zero for any $k\geq 1$ because $\cExt{k-1}(\sE,\cT)$ is torsion sheaf (being supported on $\supp(\cT)$) and $\cExt{k}(\sE,\sE)$ is a torsion-free sheaf by what already proved (using that $\sE$ is quasi-locally free), we get the  exact sequence
\begin{equation}\label{E:exExt}
0 \to \cExt{k}(\sE, \sE)\xrightarrow{p_*}\cExt{k}(\sE, \sF)\xrightarrow{q_*} \cExt{k}(\sE, \cT)\rightarrow 0 \quad \text{ for any } k\geq 1. 
\end{equation}
Consider now the exact sequence \eqref{E:seqstalk}. Using the exact sequence $0\to   \cO_{X,p}\to  \cI_{n_ip}\to \cO_{n_ip}\to 0$
and the vanishing $\Ext{k}(\cO_{X,p}, M)=0$ for any $k\geq 1$ and any module $M$ over $\cO_{X,p}$, we get the isomorphism
$$
\Ext{k}(M, \cI_{n_ip})\xrightarrow{\cong} \Ext{k}(M, \cO_{n_ip}) \quad \text{ for any } k\geq 1 \text{ and any $\cO_{X,p}$-module } M.  
$$
Applying this isomorphism to the exact sequence \eqref{E:seqstalk}, we deduce that 
$$
\Tors(\Ext{k}(M, \sF_p))\twoheadrightarrow \Ext{k}(M, \cT_p) \quad \text{ is surjective for any } k\geq 1 \text{ and any $\cO_{X,p}$-module } M,
$$
where $\Tors(-)$ denote the torsion submodule. 
This implies that 
\begin{equation}\label{E:surExt}
\cT(\cExt{k}(\sE, \sF))\stackrel{q_*}{\twoheadrightarrow} \cExt{k}(\sE, \cT) \quad \text{ is surjective for any } k\geq 1,  
\end{equation}
where $\cT(-)$ denote the torsion subsheaf. Combining \eqref{E:exExt} and \eqref{E:surExt},  and recalling that $\cExt{k}(\sE, \sE)$ is locally free and $\cExt{k}(\sE, \cT)$ is torsion, we deduce that 
\begin{equation}\label{E:splitExt}
\cExt{k}(\sE, \sF)= \cExt{k}(\sE, \sE)\oplus \cExt{k}(\sE, \cT) \quad \text{ and } \quad \cT(\cExt{k}(\sE, \sF))= \cExt{k}(\sE, \cT). 
\end{equation}
This implies that $ \cExt{k}(\sE, \sE)$ is the reflexive hull of $\cExt{k}(\sE, \sF)$, and Claim 2 is proved. 
\end{proof}

\begin{lemma}\label{L:locExt}
Let $\wh \cO_{X,p}$ be the completion of the local ring of a ribbon $X$ at a closed point $p$. Consider the ring $\wh \cO_{\Xred,p}$ and, for any $n\geq 0$, the ideal $\wh \cI_{np}$ and the ring $\wh \cO_{np}$, all of them considered as $\wh \cO_{X,p}$-modules. Then we have
\begin{enumerate}[(i)]
\item \label{L:locExt1} $\Ext{k}(\wh \cO_{\Xred,p}, \wh \cO_{\Xred,p})=\wh \cO_{\Xred,p}$ for any $k\geq 0$.
\item \label{L:locExt2} $\Ext{k}(\wh \cO_{\Xred,p},\wh \cI_{np})=
\begin{cases}
\wh \cO_{\Xred,p} &\text{  if } k= 0,\\
\wh \cO_{np} & \text{ if } k>0.
\end{cases}$
\item \label{L:locExt3} $\Ext{k}(\wh \cI_{np}, \wh \cO_{\Xred,p})=
\begin{cases}
\wh \cO_{\Xred,p} &\text{  if } k= 0,\\
\wh \cO_{np} & \text{ if } k>0.
\end{cases}$
\item \label{L:locExt4} $\Ext{k}(\wh \cI_{mp}, \wh \cI_{np})=
\begin{cases}
\wh \cI_{\max\{n,m\}p} &\text{  if } k= 0,\\
\wh \cO_{\min\{n,m\}p}^{\oplus 2} & \text{ if } k>0.
\end{cases}$
\end{enumerate}
\end{lemma}
\begin{proof}
Recall the presentation of the various rings and modules given in \eqref{E:pres-mod} and \eqref{E:pres-Inp}.  

The modules $ \wh \cO_{\Xred,p}$ and $\wh \cI_{np}$ admit the following resolutions via free modules
\begin{equation}\label{E:res-OXred}
\ldots \xrightarrow{\cdot \epsilon}   \cO_{X,p} \xrightarrow{\cdot \epsilon} \cO_{X,p} \xrightarrow{\cdot \epsilon} \cO_{X,p} \twoheadrightarrow  \cO_{\Xred,p},
\end{equation}
\begin{equation}\label{E:res-Inp}
\begin{aligned}
\ldots \xrightarrow{\begin{pmatrix} \epsilon & s^n\\0 & -\epsilon \end{pmatrix}}   \cO_{X,p}^{\oplus 2} \xrightarrow{\begin{pmatrix} \epsilon & s^n\\0 & -\epsilon \end{pmatrix}} \cO_{X,p}^{\oplus 2} \xrightarrow{\begin{pmatrix} \epsilon & s^n\\0 & -\epsilon \end{pmatrix}} \cO_{X,p}^{\oplus 2} & \twoheadrightarrow \wh \cI_{np},\\
\begin{pmatrix} a\\b\end{pmatrix} & \mapsto af+be.
\end{aligned}
\end{equation}
We now use the above resolutions to compute the desired Ext groups:

$\bullet$ 
$
\Ext{k}(\wh \cO_{\Xred,p}, \wh \cO_{\Xred,p})=H^k(\cO_{\Xred,p} \xrightarrow{\cdot \epsilon=0} \cO_{\Xred,p} \xrightarrow{\cdot \epsilon=0} \cO_{\Xred,p} \xrightarrow{\cdot \epsilon=0} \ldots)=\cO_{\Xred,p}
$ 
for any $k\geq 0$.

$\bullet$ 
$
\Ext{k}( \wh \cO_{\Xred,p}, \wh \cI_{np})=H^k(\wh \cI_{np} \xrightarrow{\cdot \epsilon} \wh \cI_{np} \xrightarrow{\cdot \epsilon} \wh \cI_{np} \xrightarrow{\cdot \epsilon} \ldots).$ 
The maps in the above complex are equal to 
$$
\begin{aligned}
\cdot \epsilon: \wh \cI_{np}& \to \wh \cI_{np}\\
a(s)e+b(s)f& \mapsto \epsilon(a(s)e+b(s)f)=s^na(s)f. 
\end{aligned}
$$ 
Therefore, the image and kernel of the above morphism are given by 
$$
\Im(\cdot \epsilon)=\{s^na(s)f\: : a(s)\in k[[s]]\}\subseteq \ker{\cdot \epsilon}=\{b(s)f\: : b(s)\in k[[s]]\}\subset \wh \cI_{np}.
$$
From this, part \eqref{L:locExt2} follows. 

$\bullet$ $\Ext{k}(\wh \cI_{np}, \wh \cO_{\Xred,p})=H^k\left( \cO_{\Xred,p}^{\oplus 2}  \xrightarrow{\begin{pmatrix} 0 & 0\\s^n & 0 \end{pmatrix}} \cO_{\Xred,p}^{\oplus 2}  \xrightarrow{\begin{pmatrix} 0 & 0\\s^n & 0 \end{pmatrix}} \cO_{\Xred,p}^{\oplus 2}  \xrightarrow{\begin{pmatrix} 0 & 0\\s^n & 0 \end{pmatrix}}  \ldots \right)$.

The image and kernel of the morphism appearing in the above complex are equal to 
$$
\Im(\begin{pmatrix} 0 & 0\\s^n & 0 \end{pmatrix})=\langle \begin{pmatrix} 0 \\ s^n\end{pmatrix} \rangle \subseteq \ker{\begin{pmatrix} 0 & 0\\s^n & 0 \end{pmatrix}}=\langle \begin{pmatrix} 0 \\ 1\end{pmatrix} \rangle
$$
From this, part \eqref{L:locExt3} follows. 

$\bullet$ $\Ext{k}(\wh \cI_{np}, \wh \cI_{mp})=H^k\left( \wh \cI_{mp}^{\oplus 2}  \xrightarrow{\begin{pmatrix} \epsilon & 0\\s^n & -\epsilon \end{pmatrix}}  \wh \cI_{mp}^{\oplus 2}  \xrightarrow{\begin{pmatrix} \epsilon & 0\\s^n & -\epsilon \end{pmatrix}}  \wh \cI_{mp}^{\oplus 2}  \xrightarrow{\begin{pmatrix} \epsilon & 0\\s^n & -\epsilon \end{pmatrix}}  \ldots \right)$.
 
 The map in the above complex is given by 
$$
\begin{aligned}
\begin{pmatrix} \epsilon & 0\\s^n & -\epsilon \end{pmatrix}: \cI_{mp}^{\oplus 2} & \longrightarrow \cI_{mp}^{\oplus 2}\\
\begin{pmatrix} a_1(s)e+b_1(s) f \\ a_2(s) e+b_2(s) f\end{pmatrix} & \mapsto \begin{pmatrix} a_1(s)s^m f \\ a_1(s)s^n e+[s^nb_1(s)-s^ma_2(s)]f\end{pmatrix}
\end{aligned}
$$

The image and kernel of the morphism appearing in the above complex are equal to 
$$
\begin{sis}
\Im(\begin{pmatrix} \epsilon & 0\\s^n & -\epsilon \end{pmatrix})=\langle \begin{pmatrix} s^m f \\ s^n e\end{pmatrix}, s^{\min\{n,m\}} \begin{pmatrix} 0 \\ f \end{pmatrix} \rangle, \\
 \ker{\begin{pmatrix} \epsilon & 0\\s^n & -\epsilon \end{pmatrix}}=
 \begin{cases}
 \langle  \begin{pmatrix}s^{m-n} e\\ f \end{pmatrix}, \begin{pmatrix} 0 \\ f\end{pmatrix} \rangle\cong \wh \cI_{mp} & \text{ if } n\leq m, \\
  \langle  \begin{pmatrix} e\\ s^{n-m} f \end{pmatrix}, \begin{pmatrix} 0 \\ f\end{pmatrix} \rangle \cong \wh \cI_{np}& \text{ if } n\geq m, \\

 \end{cases}
 \end{sis}
$$
From this, part \eqref{L:locExt4} follows. 
\end{proof}

\begin{cor}\label{C:Ext1}
Let $\sF\in \cM_X(r_\bullet;d_\bullet)$. Then we have that:
\begin{enumerate}[(i)]
\item $\sF$ is a generalized vector bundle (i.e. if $r_0=r_1$) if and only if $\cExt{k}(\sF,\sF)$ is a torsion sheaf for any $k\geq 1$.
\item  $\sF$ is quasi-locally free (i.e. if $\iota(\sF)=0$) if and only if $\cExt{k}(\sF,\sF)$ is a torsion-free sheaf for any $k\geq 1$.
\end{enumerate}
\end{cor}

\begin{cor}\label{C:2Tan}
Let $\sF\in \cM_X(r_\bullet;d_\bullet)$. The inclusion 
$$T_{\sF}\cM_X(r_\bullet;d_\bullet)\subseteq T_{\sF}\cM_X(R,D)$$
has codimension equal to 
$$
h^0(\Xred, \cEnd\left(\left(\frac{\sF^{(1)}}{\sN\sF}\right)^{**}\right)\otimes \sN^{-1})+\iota(\sF)a(\sF)+\sum_{p\in \supp(\cT(\sF_{|\Xred}))}(2i-1)n_i(\sF,p).
$$
\end{cor}
\begin{proof}
This follows from Theorem \ref{T:Ext1} and Corollary \ref{C:Tan-lt}, using the identification $ T_{\sF}\cM_X(R,D)=\Ext{1}(\sF,\sF)$ and the exact sequence \eqref{E:Ext}.
\end{proof}

\begin{cor}\label{C:Tan-qlf}
Let $\sF\in \cM_X(r_\bullet;d_\bullet)$ and assume that $\sF$ is quasi-locally free. Then we have an exact sequence
\begin{equation}\label{E:seq-qlf}
0\to T_{\sF}\cM_X(r_\bullet;d_\bullet)\to T_{\sF}\cM_X(R,D)\to H^0(\Xred,\cEnd\left(\frac{\sF^{(1)}}{\sN\sF}\right)\otimes \sN^{-1})\to 0.
\end{equation}
In particular, if $\sF$ is of rigid type (i.e. if $r_0=r_1$ or $r_1+1$) then we have that 
\begin{enumerate}[(i)]
\item \label{C:Tan-qlf1} If $r_0=r_1$ (i.e. if $\sF$ is a locally free sheaf) then $\cM_X(R,D)$ is smooth at $\sF$. 
\item \label{C:Tan-qlf2} If $r_0=r_1+1$ then the underlying reduced stack $\cM_X(R,D)_{\red}$ is smooth at $\sF$ and we have the identification
$$
\frac{T_\sF \cM_X(R,D)}{T_\sF \cM_X(R,D)_{\red}}=H^0(\Xred,\sN^{-1}).
$$
\end{enumerate}
\end{cor}
Part \eqref{C:Tan-qlf1} is well-known (and it also follows from the fact that $\Ext{2}(\sF,\sF) =0$ if $\sF$ is locally free), while part  \eqref{C:Tan-qlf2} of the above Corollary follows from \cite[Thm. E]{DR2} and \cite[Cor. 6.2.2]{DR1} (see also \cite[Sec. 1.1.2]{DR5}).
\begin{proof}
This follows from Corollary \ref{C:2Tan}, using that  if $\sF$ is quasi locally free then 
\begin{itemize}
\item $\displaystyle \frac{\sF^{(1)}}{\sN\sF}$ is torsion-free on $\Xred$ because it is a subsheaf of the torsion-free sheaf $\sF_{|\Xred}$;
\item  $\iota(\sF)=0$ and $\cT(\sF_{|\Xred})=0$.
\end{itemize}

For the last statement, observe that  if $\sF\in \cM_X(r_\bullet; d_\bullet)$ is of rigid type then $\cM_X(r_\bullet;d_\bullet)$ is open in $\cM_X(R,D)$ by Theorem \ref{T:specia} and hence it  coincides 
with the restriction of $\cM_X(R,D)_{\red}$ on this open subset. Therefore, we conclude using \eqref{E:seq-qlf} and the fact that 
\begin{itemize}
\item If $r_0=r_1$ then $\frac{\sF^{(1)}}{\sN\sF}=0$;
\item If $r_0=r_1+1$ then $\frac{\sF^{(1)}}{\sN\sF}$ is a line bundle and hence $\cEnd\left(\frac{\sF^{(1)}}{\sN\sF}\right)=\cO_{\Xred}$. 
\end{itemize}
\end{proof}

\begin{cor}\label{C:Tan-gvb}
Let $\sF\in \cM_X(r,r;d_\bullet)$, i.e. $\sF$ is a generalized vector bundle of rank $r$. The inclusion 
$$T_{\sF}\cM_X(r_\bullet;d_\bullet)\subseteq T_{\sF}\cM_X(R,D)$$
has codimension equal to 
$$
\sum_{p\in \supp(\cT(\sF_{|\Xred}))}(2i-1)n_i(\sF,p)\geq \iota(\sF).
$$
Furthermore, if $\sF$ is general in  $\cM_X(r,r;d_\bullet)$, the the above codimension is equal to $\iota(\sF)$. 
\end{cor}
This was proved for generalized line bundles (in which case the codimension is always equal to $\iota(\sF)$) in \cite[Prop. 4.1]{CK}.
\begin{proof}
The first part follows from Corollary \ref{C:2Tan} using that $a(\sF)=0$ and  that $\displaystyle \frac{\sF^{(1)}}{\sN\sF}$ is a torsion sheaf. 
The second part follows from the first one and Corollary \ref{C:gengvb}. 
\end{proof}

\end{document}